\documentclass[12 pt]{amsart}
\usepackage{amssymb, amsmath, amsfonts, amsthm, bbm, graphics}
\usepackage[hmargin=1 in, vmargin = 1 in]{geometry}
\usepackage{tikz-cd}
\usepackage{hyperref}
\usepackage{enumitem}
\usepackage[all]{xy}
\usepackage{stmaryrd}
\usepackage{booktabs}

\usepackage{mathrsfs}
\usepackage[most]{tcolorbox}
\usepackage[most]{tcolorbox}
\usepackage{xcolor}
 \usepackage[T1]{fontenc}
\usepackage{libertinus}
\tcbset{
  myhighlight/.style={
    colback=yellow!20,
    colframe=yellow!50!black,
    boxrule=0.4pt,
    arc=1mm,
    left=1mm,right=1mm,top=1mm,bottom=1mm
  }
}
\usepackage{mathabx,epsfig}

\usepackage{amsthm}

\usepackage{todonotes}
\numberwithin{equation}{section}
\newtheorem{theorem}{Theorem}[section]        
\newtheorem{proposition}[theorem]{Proposition}

\newtheorem{lemma}[theorem]{Lemma}
\newtheorem{corollary}[theorem]{Corollary}
\newtheorem{example}[theorem]{Example}

\theoremstyle{definition}

\newtheorem{definition}[theorem]{Definition}
\newtheorem{remark}[theorem]{Remark}

\newcommand{\restr}[2]{\left.\kern-\nulldelimiterspace #1 \vphantom{\big|}\right|_{#2}}

\newcommand{\Bun}{\mathrm{Bun}}

\DeclareMathOperator{\Aut}{Aut}
\DeclareMathOperator{\Pic}{Pic}

\DeclareMathOperator{\Spec}{Spec}

\newcommand{\im}{\operatorname{im}}

\DeclareMathOperator{\supp}{supp}

\newcommand{\GL}{\mathrm{GL}}
\newcommand{\PGL}{\mathrm{PGL}}
\newcommand{\mat}[4]{\begin{pmatrix} #1 & #2 \\ #3 & #4 \end{pmatrix}}
\newcommand{\smat}[4]{\left(\begin{smallmatrix} #1 & #2 \\ #3 & #4 \end{smallmatrix}\right)}

\usepackage[most]{tcolorbox}
\usepackage{xcolor}

\tcbset{
  myhighlight/.style={
    colback=yellow!20,
    colframe=yellow!50!black,
    boxrule=0.4pt,
    arc=1mm,
    left=1mm,right=1mm,top=1mm,bottom=1mm
  }
}

\renewcommand{\AA}{\mathbb{A}}

\newcommand{\CC}{\mathbb{C}}

\newcommand{\FF}{\mathbb{F}}

\newcommand{\PP}{\mathbb{P}}

\newcommand{\ZZ}{\mathbb{Z}}

\newcommand{\cE}{\mathcal{E}}

\newcommand{\cH}{\mathcal{H}}

\newcommand{\cL}{\mathcal{L}}

\newcommand{\cO}{\mathcal{O}}
\newcommand{\cP}{\mathcal{P}}

\newcommand{\fg}{\mathfrak{g}}

\theoremstyle{definition}

\theoremstyle{definition}

\usepackage{amsmath,amssymb}

\makeatletter
\newcommand{\xRrightarrow}[2][]{\ext@arrow 0359\Rrightarrowfill@{#1}{#2}}
\newcommand{\Rrightarrowfill@}{\arrowfill@\equiv\equiv\Rrightarrow}
\newcommand{\xLleftarrow}[2][]{\ext@arrow 3095\Lleftarrowfill@{#1}{#2}}
\newcommand{\Lleftarrowfill@}{\arrowfill@\Lleftarrow\equiv\equiv}

\makeatother

\begin{document}

\title{Graphs of Hecke operators in mixed ramification}
\author{Rudrendra Kashyap}
\address{Rudrendra Kashyap, University of Pittsburgh, Pittsburgh,
PA, USA}
\email{ruk26@pitt.edu}
\author{Vladyslav Zveryk}
\address{Vladyslav Zveryk, Yale University, New Haven,
CT, USA}
\email{zverik.vladislav@gmail.com}
\maketitle

\begin{abstract}
    We study Hecke operators on moduli spaces of ramified $G$-bundles using the combinatorial language of Hecke graphs. We introduce a general notion of $\cH$-ramification in the spirit of parahoric ramification, which depends on a choice of a divisor and subgroups of $G$ at every point of the divisor. Building on our previous work, we prove that, under mild regularity conditions, the action of a Hecke operator in the deep cusp of $\Bun_G$ in a highly complex ramification mimics an action in a much simpler ramification. This reduces the study to a smaller number of cases which, in particular, involve divisors supported at no more than two points. We demonstrate our methods by computing various examples for $G=\PGL_2$ and computing the dimensions of spaces of Hecke eigenforms for generic eigenvalues. We connect the obtained dimension formulas with the known results from the theory of Eisenstein series.
\end{abstract}

\tableofcontents

\section{Introduction}
Let \(X\) be a smooth projective geometrically connected curve over a finite
field \(k\), and let \(G\) be a connected split reductive group. We study principal
\(G\)-bundles on \(X\), possibly equipped with level structure along an
effective divisor. A Hecke operator is obtained by choosing a closed point
\(x\in |X|\) and a modification type at \(x\); it acts on functions on the set
of isomorphism classes of such bundles with level structure. Equivalently, on
\(k\)-points it is the adjacency operator of a directed graph whose vertices are
ramified \(G\)-bundles and whose edges are Hecke modifications. Ramification
therefore changes the graph by changing the vertices, and, if \(x\) belongs to
the ramification divisor, also by changing the local modification condition. The
purpose is to describe this change in the Harder--Narasimhan cusp.

In \cite{KZ26}, we treated principal ramification, the function-field analogue
of \(\Gamma(n)\)-ramification on modular forms. In this case the level datum is a
trivialization of the bundle over the ramification divisor. Hence an
automorphism preserving the level datum must restrict to the identity on that
finite divisor. The main geometric input was a computation, deep in the
Harder--Narasimhan cusp, of the image on the level divisor of the automorphism
group of the underlying bundle. When the Hecke point is disjoint from the
ramification divisor, this gives a disjoint covering of the unramified cuspidal
graph. When the Hecke point is ramified, the local Hecke correspondence has to
be computed separately. Combining these cusp computations with a propagation
argument for infinite graphs gave bounds, and formulas exact for all but
finitely many eigenvalues, for Hecke eigenspaces for \(G=\mathrm{PGL}_2\).

The present paper replaces principal ramification by general \(\cH\)-ramification.
At each ramified point \(y\), one fixes a subgroup \(H^y\subset G\) and keeps a
reduction of the restricted bundle to \(H^y\). This includes principal,
unipotent, and parahoric-type examples. Such a reduction may be
preserved by nontrivial automorphisms of the restricted bundle. We compute how
these automorphisms affect the fibers of forgetful maps between ramified moduli
spaces. In the cusp, the answer separates a unipotent contribution, which is
uniform, from a torus contribution, which is constrained by the level structure
retained at the other ramified points. This leads to a regularity condition
under which the forgetful map is a covering of Hecke graphs, and to a torus
obstruction to this covering being disjoint. We then use this description to
reduce cusp computations to basic ramification types, compute local ramified
Hecke correspondences for \(\mathrm{GL}_n\), and apply the results to
\(\mathrm{PGL}_2\).

\subsection{Main results}

Let 
$
D=\sum_{y\in X}d_y[y]
$ 
be an effective divisor decomposed into a sum $D=D_1+D_2$ of effective divisors with disjoint support, and assume that the Hecke point \(x\) is not in the support of \(D_1\). We study the map
\[
p_{D,D_2}:\operatorname{Bun}^{\cH}_{G,D}(k)\longrightarrow \operatorname{Bun}^{\cH}_{G,D_2}(k),
\]
which forgets the level structure over \(D_1\).

Our first result computes the fibers of this map deep in the Harder--Narasimhan cusp. Let \(P_G\) be a \(G\)-bundle whose Harder--Narasimhan reduction is a \(B\)-reduction, and assume that \(P_G\) lies sufficiently deep in the \(D\)-cusp. Fix an \(\cH\)-level structure over \(D_2\). After choosing a trivialization of $\cP_G$ over $D_2$ sending $\cP_B$ to $B$, we can represent this level structure by an element \(\tau\in G(\mathcal O_{D_2})\), where $\cO_{D_2}$ is the coordinate ring of $D_2$. We define
\begin{align*}
T_{D_2}(\tau)
&:=
T(k)\cap
\bigcap_{y\in \operatorname{supp}(D_2)}
\tau_y H^y(\mathcal O_y/\pi_y^{d_y})\tau_y^{-1}
U(\mathcal O_y/\pi_y^{d_y}),\\
\cH(\cO_D)&:=\prod_{y\in\supp D}H^y(\cO_{y}/\pi_{y}^{d_y}).
\end{align*}

Then we get the following result (Theorem \ref{t:automorphisms_reductive_general} and Corollary \ref{cor:fiber_size_reductive}).
\begin{theorem}
The image of the restriction map
\[
\operatorname{Aut}(P_G,\varphi_{D_2})
\longrightarrow G(\mathcal O_{D_1})
\]
under the chosen trivialization is
\[
T_{D_2}(\tau)\ltimes U(\mathcal O_{D_1})
\]
when \(D_2\neq 0\), and is \(T(k)\ltimes U(\mathcal O_{D_1})\) when \(D_2=0\). Therefore
\[
p_{D,D_2}^{-1}(P_G,\varphi_{D_2})
\simeq
\cH(\mathcal O_{D_1})\backslash G(\mathcal O_{D_1})/
\bigl(T_{D_2}(\tau)U(\mathcal O_{D_1})\bigr).
\]
\end{theorem}

This formula is the main geometric input of the paper.
In particular, it implies that the fiber size is not always constant. To handle this, we introduce a regularity condition. We say that the ramification datum is regular over \(D_2\) if \(T_{D_2}(\tau)\) does not depend on \(\tau\). This condition holds in particular if one of the retained subgroups \(H^y\) is unipotent; in that case the corresponding torus group is trivial.

Under this regularity condition, the forgetful map becomes a covering map of Hecke graphs in the cusp. More precisely, our main result is the following (Theorem \ref{t:main_theorem}).
\begin{theorem}
Deep in the cusp, the map of graphs
\[
p_{D,D_2}:\Gamma_{D,\cH}^{\mathrm{cusp}}\longrightarrow \Gamma_{D_2,\cH}^{\mathrm{cusp}}
\]
forgetting the level structure at $D_1$ is a covering map of directed graphs. Moreover, if ramification is regular at $D_2':=D_2-d_x[x]$, then the monodromy of this covering is contained in $T_{D_2'}(\tau)/T_{D_2}(\tau)$. 

In particular, if $T_{D_2'}(\tau)=T_{D_2}(\tau)$, then the covering is disjoint. This assumption is satisfied if there exists a point $y\in\supp D_2$, $y\ne x$ such that $H^y$ is unipotent.
\end{theorem}

 This result allows us to significantly reduce the complexity of ramification. Another result makes it possible to change the ramification subgroups $H^y$ (Proposition \ref{p:change_of_ramification}). One particular case of it is the following.
 \begin{proposition}
     Choose two ramification data $(D,\cH_1)$ and $(D,\cH_2)$ with $H^y_1=1$ for all $y\ne x$ and $H_1^x=H_2^x$. The second ramification includes less structure, hence the vertices of the Hecke graph for $(D,\cH_2)$ are obtained from identifying certain vertices of the Hecke graph for $(D,\cH_1)$. This correspondence on vertices induces a bijection of edges.
 \end{proposition}

As a computational consequence of these results, to know all cusp graphs for arbitrary ramification it is enough
to compute three basic families (Section~\ref{ss:summary_computations}). Fix the Hecke point \(x\). The required families are
\begin{enumerate}
    \item $D=0$,
    \item $\operatorname{supp}D=\{x\},\ H^x=1$,
    \item $\operatorname{supp}D=\{x,y\},\ H^y=1,\ \text{with arbitrary }H^x$.
\end{enumerate}

We also describe the local Hecke correspondences for \(G=\operatorname{GL}_n\) explicitly when the Hecke point itself is ramified. Let \(\omega_r\) be the \(r\)-th fundamental coweight, and let \(P\) be the standard parabolic of block type \((r,n-r)\), with unipotent radical \(U\). We prove the following (Theorem \ref{t:hecke_gln_cases} and Remark \ref{r:strengthening_of_general_connections}).
\begin{theorem}\
    \begin{enumerate}[label=(\roman*)]
        \item If $U\subset H^x\subset P$, then the correspondence gives \(q^{r(n-r)\deg(x)}\) distinct subbundles, and each of them has a unique compatible \(H^x\)-level structure.
\item If $U^{-}\subset H^x\subset P^{-}$ or \(H^x\) is contained in the Levi subgroup of $P$, then the underlying subbundle is unique, but the edges are still parametrized by an explicit affine space of size \(q^{r(n-r)\deg(x)}\). 
    \end{enumerate}
\end{theorem}
Thus the same local number appears, but it is seen in two different ways: either as many subbundles with one level structure each, or as one subbundle with many compatible edges. We also describe how to get a graph for $\PGL_n$ from the graph for $\GL_n$: under reasonable assumptions on $\cH$ (actually, only $H^x$), it is obtained by gluing vertices in the graph for $\GL_n$ and keeping all the edges. See Proposition \ref{p:from_GL_to_PGL}.

Finally, we apply the geometric results to eigenspaces for \(G=\operatorname{PGL}_2\). Let \(\Gamma\) be the Hecke graph, and let \(\dim_\lambda \Gamma\) be the dimension of the \(\lambda\)-eigenspace of its adjacency operator. 

\begin{theorem}
Write \(r=\deg(x)\). Let $\lambda \neq 0$. Then there exists a finite subgraph \(\Gamma'\) of $\Gamma$ such that:
    \begin{enumerate}[label=(\roman*)]
        \item If \(x\notin \operatorname{supp}(D)\), then
\[
r|\operatorname{Pic}^0(X)(k)|
\cdot
\left|
\cH(\mathcal O_D)\backslash G(\mathcal O_D)/
(T(k)\ltimes U(\mathcal O_D))
\right|
\leq
\dim_\lambda \Gamma
\]
and
\[
\dim_\lambda \Gamma
\leq
\dim_\lambda \Gamma'
+
r|\operatorname{Pic}^0(X)(k)|
\cdot
\left|
\cH(\mathcal O_D)\backslash G(\mathcal O_D)/
T(k)U(\mathcal O_D)
\right|.
\]
\item If $D=d[x]$ and $H^x\in \{1,U\}$, then
\[
r|\operatorname{Pic}^0(X)(k)|
\frac{(q^r-1)q^{r(d-1)}}{q-1}
\leq
\dim_\lambda \Gamma
\leq
\dim_\lambda \Gamma'
+
r|\operatorname{Pic}^0(X)(k)|
\frac{(q^r-1)q^{r(d-1)}}{q-1}.
\]
For general $D=D'+d[x]$, the same number appearing both in the lower and upper bounds must be multiplied by $|\cH(\cO_{D'})\backslash G(\cO_{D'})/U(\cO_{D'})|$
\item If $D=d[x]$ and $H^x=B$, then
\[
r|\operatorname{Pic}^0(X)(k)|
\leq
\dim_\lambda \Gamma
\leq
\dim_\lambda \Gamma'
+
r|\operatorname{Pic}^0(X)(k)|.
\]
\end{enumerate}
These bounds are equalities for all but finitely many \(\lambda\).
\end{theorem}
These are Theorem \ref{t:PGL2_ramified_at_x} and Theorem \ref{t:H^x=U_dimensions} enhanced with the $H^x=1$ case computed in \cite{KZ26}.
 The same propagation argument also determines the generalized eigenspaces. For every graph in the preceding theorem, every $m \geq 1$, and every $\lambda$ outside a finite exceptional set, one has \[\dim_{\CC}\ker\bigl((\Phi_x-\lambda\,\mathrm{id})^m\bigr)
=
m\,\dim_{\CC}\ker(\Phi_x-\lambda\,\mathrm{id}).\]
Equivalently, every successive generalized-eigenspace quotient has the same dimension as the ordinary eigenspace; see
Theorem~\ref{t:generic_generalized_eigenspaces}.

The last part of the paper gives explicit examples for \(\operatorname{PGL}_2\), especially over \(\mathbb P^1\) with several ramified points. These examples show how the graphs change when the ramification \(\cH\) changes, while the Hecke edges remain computable. They also compare full, unipotent, Borel, and torus ramification.

The Hecke correspondences studied here are the same geometric correspondences that define Hecke functors in the geometric Langlands program. For a general introduction to this topic, see \cite{gaitsgory2016recentprogressgeometriclanglands}. We do not construct Hecke eigensheaves in this paper. Our results are about graphs, automorphic functions, and the combinatorics of Hecke operators over finite fields. Still, they give explicit models for ramified Hecke correspondences with level structure. Known constructions of ramified eigensheaves in positive characteristic include Drinfeld's rank two tame unipotent case \cite{Dri87}, Heinloth's rank three tame unipotent case \cite{Hei04}, the tame unipotent case on \(\mathbb P^1\) minus two points \cite{AB09, Bez16}, as explained in \cite[Theorem~4.16]{BZN18}, the rank two four-point case \cite{Bos22}, Kloosterman sheaves and their generalizations \cite{HNY13, Yun16}, and Yun's constructions for tamely ramified and rank two hypergeometric local systems \cite{Yun14b}. From this viewpoint, the graphs in this paper are concrete automorphic objects attached to ramified geometric Hecke correspondences.

\subsection*{Overview}

In Section 2, we fix notation for ramification data, Hecke operators, and Hecke graphs.

In Section 3, we compare the adelic and geometric descriptions of \(H\)-ramified bundles. We describe Hecke correspondences for \(\operatorname{GL}_n\) and for a general reductive group, and we compute the local edge behavior needed later.

In Section 4, we prove the general cusp results. We compute the image of global automorphisms on the retained divisor, describe the fibers of the forgetful map, define regular ramification, and prove the covering theorem for Hecke graphs in the cusp. We also describe the torus obstruction to the covering being disjoint.

In Section 5, we apply the general theory to \(\operatorname{PGL}_2\). We compute ordinary and generalized eigenspaces when the Hecke point is unramified, then treat \(U\)- and \(B\)-ramification at the Hecke point. We finish with explicit examples for several ramified points on \(\mathbb P^1\).

In Section 6, we specialize to $G=\PGL_2$ and compare the generic graph-theoretic eigenspace dimension, for a Hecke point away from the ramification divisor, with the corresponding Eisenstein-series parameter count.

\subsection*{Acknowledgements}
The first author was partially supported by NSF grant DMS-2402553. The second author was partially supported by NSF grant DMS-2501558.
\section{Graphs of Hecke operators and notation}\label{s:preliminaries_on_Hecke_graphs}
For expositions, see \cite{Lor13,Alv19,AB24, KZ26}. Fix the following notation, which is compatible with \cite{KZ26} and will be used throughout the paper:
\begin{itemize}
    \item $X$ is a smooth projective geometrically connected curve $X$ over a finite field $k=\FF_q$.
    \item $|X|$ is the set of closed points in $X$. Choose $x\in |X|$.
    \item $F=k(X)$ is the function field, $\cO_y$ is the completion of the local ring of $X$ at $y$, $F_y$ is its fraction field. Let $\pi_y$ denote any uniformizer of $\cO_y$ and $k_x:=\cO_{x}/(\pi_x)$ the residue field of $x$.
    \item $\AA=\prod'_{y\in |X|}F_y$ is the adele ring of $X$.
    \item $D=\sum_y d_y[y]$ is an effective divisor over $X$. We often write $D=D'+d_x[x]$, where $D'=\sum_{y\ne x}d_y[y]$.
    \item $D=D_1+D_2$ for some effective divisors $D_1,D_2$ with $\supp D_1\cap\supp D_2=\varnothing$ and $x\notin\supp D_1$.
    \item $G$ is a connected split reductive group over $k$, $B$ is a Borel subgroup and $T$ is a maximal torus in $B$.
    \item For each point of $\supp D$, an algebraic subgroup $H^y$ of $G$ is chosen. We write $\cH=\{H^y\}_{y\in\supp D}$.
    \item $\Bun_{G,D}^\cH(k)$ is the set of isomorphism classes of principal $G$-bundles with $\cH$-ramification at $D$. A map $$p_{D,D_2}^\cH:\Bun_{G,D}^\cH(k)\to \Bun_{G,D_2}^\cH(k)$$ is the canonical projection forgetting the ramification at $D_1$. We set $p_D:=p_{D,0}$.
    \item $\deg D:=\sum_y d_y\deg(y)$ and $D_{\mathrm{red}}:=\sum_{y\in\operatorname{supp} D}1\cdot [y]$ is the reduced divisor.
    \item $\cO_D:=\prod_{y\in |X|}\cO_{y}/\pi_{y}^{d_y}$ is the ring of functions on $D$. We also set
    $$
    \cH(\cO_D):=\prod_{y\in\supp D}H^y(\cO_{d_y[y]})=\prod_{y\in\supp D}H^y(\cO_{y}/\pi_{y}^{d_y}).
    $$
    
    \item $K_\cH(D):=\prod_{y\in |X|}K_\cH(D)_y$ with
    $$
    K_\cH(D)_y:=p^{-1}(H^y(\cO_{d_y[y]}))\subset G(\cO_{y}),
    $$
    where $p:G(\cO_y)\to G(\cO_{d_y[y]})$ is the canonical projection.
    
    \noindent
    {\bfseries Note:} $\cH$ will be always considered present and will often be suppressed from the notation.

    \item $\mu$ is a dominant cocharacter of $T$ and $\Delta^\mu$ is the corresponding element of $$K_\cH(D)\backslash G(\AA)/K_\cH(D).$$
    \item $\Phi^\mu_{D,x}$ is the Hecke operator corresponding to $\Delta^\mu$.
    \item $\Gamma^\mu_{D,x}$ is the graph of $\Phi^\mu_{D,x}$.
    \item $\dim_\lambda\Gamma^\mu_{D,x}$ denotes the dimension of the $\lambda$-eigenspace of $\Phi^\mu_{D,x}$.
\end{itemize}

\begin{definition}
    Let $D$ be an effective divisor on $X$. A {\bfseries ramification datum over $D$} is a choice of an algebraic subgroup $H^y$ of $G$ for any point $y\in\supp D$. We will write $\cH=\{H^y\}_{u\in\supp D}$. We define $K_\cH(D):=\prod_{y\in |X|}K_\cH(D)_y$ with
    $$
    K_\cH(D)_y:=p^{-1}(H^y(\cO_{d_y[y]}))\subset G(\cO_{y}),
    $$
    where $p$ is the canonical reduction modulo $\pi_y^{d_y}$ map
    $$
    p:G(\cO_y)\to G(\cO_{d_y[y]}).
    $$
    
    A {\bfseries $G$-bundle with an $\cH$-ramification} is an element of $G(F)\backslash G(\AA)/K_\cH(D)$. 
\end{definition}

Throughout the whole paper, a ramification datum will be chosen and $\cH$ will usually be suppressed from the notation.

For $\lambda \in \CC$, define 
$$
\dim_{\lambda}\Gamma_{D,x}^{\mu}
:=\dim_{\CC}\ker\bigl(\Phi_{D,x}^{\mu}-\lambda\,\mathrm{id}).
$$




\section{Transition between adelic and geometric sides}\label{section 2}
This section extends \cite[Section 3]{KZ26}.

First, recall the bijection
\begin{equation*}
    G(F)\backslash G(\AA)/G(\cO)\simeq \Bun_G(X)(k).
\end{equation*}
Choose a principal $G$-bundle $\cP$ on $X$ and its trivializations $\phi_x:\cP|_{\cO_x}\simeq G_{\cO_x}$ for any $x\in X$, including the generic point $\eta$. This gives a composition of isomorphisms
$$
G_F=G_{\cO_x}|_\eta\xrightarrow{\phi_x^{-1}|_\eta}\cP|_{\cO_x}|_\eta=\cP_\eta\xrightarrow{\phi_\eta}G_F.
$$
Because we got a multiplication-invariant automorphism of $G_F$, it is given by multiplication by an element $g_x\in G(F)$. These form an element $(g_x)_{x\in |X|}\in G(\AA)$. This element depends on our choices of the trivializations $\phi_x$ which are unique up to the coordinate change by $G(\cO_x)$, and $\phi_\eta$ which is unique up to $G(F)$. Therefore, the element $(g_x)_{x\in |X|}\in G(\AA)$ is well-defined as an element of $G(F)\backslash G(\AA)/G(\cO)$.

We can generalize this bijection to the ramified case. For this, we introduce the following definition:
\begin{definition}\label{def:Bcompatible}
    Let $H$ be an algebraic subgroup of $G$, $\cP_G$ be a principal $G$-bundle on $X$, $d[x]$ a divisor supported at a point $x$, and $\cP_H|_{d[x]}$ be a reduction of $\cP_G$ to $H$ over $\cO_{d[x]}$. A trivialization $\tau$ of $\cP_G$ over $\cO_x$ or $\cO_{d[x]}$ is said to be {\bfseries compatible with $\cP_H|_{d[x]}$} if $\tau(\cP_H|_{d[x]})=H\times \Spec \cO_{d[x]}$.
\end{definition}

Having a $(D,\cH)$-ramification on a principal bundle $\cP_G$ is equivalent to having the following additional structure: For each $y\in \supp D$, a reduction of $\cP_G$ to $H$ over $\cO_{d_y[y]}$. We denote the set of such data by $\Bun_{G,D}^\cH(k)$.
This was proven for $H^y=1$ in \cite[Section 2]{KZ26}. To get this statement for a general $\cH$-ramification, take a principal $G$-bundle $\cP_G$ with a fixed reduction $\cP_{H^y}$ to $H^y$ over $\cO_{d_y[y]}$ for every $y$. As above, choose a trivialization $\phi_x$ of $\cP_G$ at each point $x\in X$, but additionally assume that $\phi_y$ is compatible with $\cP_{H^y}$ for $y\in\supp D$. This gives a well-defined point of 
$$
G(F)\backslash G(\AA)/K_\cH(D)
$$
since, by definition of $K_\cH(D)$, the action of the group $K_\cH(D)$ by changing trivializations is transitive on those compatible with $\cP_{H^y}$. For the inverse map, to a point in this double quotient, associate a principal bundle as above with 
$$
\cP_{H_y}:=\phi_y^{-1}(H^y\times\Spec \cO_{d_y[y]}).
$$
These two maps are clearly mutually inverse and give the desired bijection.

\begin{remark}\label{rem:coset_conventions}

Let $\cP_G$ be a vector bundle and $D=d[x]$. We can parametrize $H^x$-level structures over $D$ in the following way. Fix a trivialization $\phi$ of $\cP_G$ over $D$. It defines a level structure $\phi^{-1}(H^x)$ on $\cP_G$. Any other trivialization has the form $g\phi$, where $g\in G(\cO_{d[x]})$. The level structure it defines is $\phi^{-1}g^{-1}H^x$. Note that elements $g_1,g_2\in G(\cO_{d[x]})$ give the same $H^x$-reduction if and only if $g_1H^x=g_2H^x$, which identifies $H^x$-reductions on $\cP_G|_{D}$ with $H^x(\cO_{d[x]})\backslash G^x(\cO_{d[x]})$.

Recall that vertices of Hecke graphs are isomorphism classes of bundles with level structures. Therefore, we need to quotient $H^x(\cO_{d[x]})\backslash G^x(\cO_{d[x]})$ by the action of $\Aut(\cP_G)$. Let $A$ be the image of $\Aut(\cP_G)$ in $\Aut(\cP_G|_{d[x]})\simeq G(\cO_{d[x]})$, where the last identification is made using $\phi$. Specifically, this map takes $a\in \Aut(\cP_G)$ to $\phi a\phi^{-1}$. It sends a trivialization $g\phi$ to $g\phi a$, and hence the level structure $\phi^{-1}g^{-1}$ to $a^{-1}\phi^{-1}g^{-1}H$. Therefore, the $H^x$-level structures on $\cP_G$ up to isomorphism are identified with
$$
H^x(\cO_{d[x]})\backslash G(\cO_{d[x]})/A.
$$

More generally, for a ramification datum $(\cH,D)$, the level structures on $\cP_G$ are parametrized by
$$
\cH(\cO_{D})\backslash G(\cO_D)/A,
$$
where $A$ is the image of $\Aut(\cP_G)$ in $\Aut(\cP_G|_D)\simeq G(\cO_D)$ under a chosen trivialization of $\cP_G$ over $D$.

We can have a description without choosing a reference trivialization. To a trivialization $\phi:\cP_G|_{D}\simeq G\times \Spec\cO_{D}$ associate the $\cH$-level structure given by $\phi^{-1}(H^y\times\Spec\cO_D)_{y\in\supp D}$. Two trivialization $\phi_1,\phi_2$ give the same level structures if and only if $\phi_2=h\phi_1$ for some $h\in \cH(\cO_D)$. Moreover, the automorphism group $\Aut(\cP_G)$ acts on these trivializations on the right by composition, so the set of $\cH$-level structures on $\cP_G$ up to equivalence is given by the double quotient
$$
\cH(\cO_D)\backslash\mathrm{Iso}(\cP_G|_D,G\times \Spec\cO_D)/\Aut(\cP_G).
$$
Choosing a reference trivialization of $\cP_G$ and identifying $\mathrm{Iso}(\cP_G|_D,G\times \Spec\cO_D)$ and $\Aut(\cP_G|_{D})$ with $G(\cO_D)$ under this trivialization, we get the above explicit description.
\end{remark}

\begin{remark}
    Methods and results of our paper apply even in a stronger setting: instead of choosing algebraic subgroups $H^y$, we can choose arbitrary subgroups of finite groups $G(\cO_{d_y[y]})$. This makes the adelic setting unchanged, and the geometric setting is assumed to use the double coset interpretation of Remark \ref{rem:coset_conventions}. However, if one wishes to define and study the stack $\Bun_{G,D}^\cH$, one needs $H^y$ to be algebraic. 
\end{remark}

\subsection{Hecke correspondences for $\GL_n$} \label{Hecke Corres.GLn}
This is an extension of \cite[Section 3]{KZ26}. 
We use the notation from Section \ref{s:preliminaries_on_Hecke_graphs}. For $\GL_n$, there is an equivalence between principal $G$-bundles and vector bundles of rank $n$. Therefore, we will utilize the language of vector bundles. Let $\omega_r$ be the $r$-th fundamental coweight of $\GL_n$ and $\Phi_{D,x}^{\omega_r}$ be the Hecke operator corresponding to it. Note that $\Delta_x = \mathrm{diag}(\pi_x I_r, I_{n-r})$ and $\Delta_y=\mathrm{id}$ for $y\ne x$. For convenience, we set $K:=K_\cH(D)$.

\begin{proposition}\label{p:Hecke_correspondence_for_vector_bundles}
    Choose $(\cE,a)\in\mathrm{Bun}_{G, D}(k)$. Edges $(\cE,a)\to(\cE',b)$ in $\mathrm{Bun}_{G, D}(k)$ for the Hecke operator $\Phi_x^{\omega_r}$ are in bijection with equivalence classes of exact sequences
$$
\{(\cE',b),\,0\to \cE'\xrightarrow{f}\cE\to k_x^{\oplus r}\to 0\}/\Aut(\cE')\times\GL(k_x^r)
$$
 such that 
    \begin{itemize}
        \item There exist $\cO_x$-trivializations $\tilde{a}$ and $\tilde{b}$ compatible with the $H^x$-reductions $a$ and $b$ such that the localization $f_x$ satisfies
$$
    \tilde{a}\circ f_x\circ \tilde{b}^{-1}\in K_x\Delta_xK_x.
$$
    \item For each $y\ne x$, the isomorphism $\cE'_y\simeq\cE_y$ induced by $f$ intertwines $a$ and $b$.
    \end{itemize}
\end{proposition}
\begin{proof}
The proof goes similarly to \cite[Proposition 3.2]{KZ26}.
\end{proof}

Write 
$$
D=D'+d_x[x],\qquad x\notin\supp D'.
$$
\begin{theorem}\label{t:hecke_gln_cases}
Let $P$ be the standard parabolic of
    block type $(r,n-r)$ and $U$ be the unipotent radical of $P$. Let also $P^-$ and $U^-$ be their transposes. Assume that $d_x\ge 1$. For a point
$(E,\phi)\in \Bun_{G,D}(k)$,  we have: 
\begin{enumerate}[label=(\roman*)]
    \item Assume that $U\subset H^x\subset P$. Then there are exactly
    $q^{r(n-r)\deg(x)}$ distinct subbundles $E'\subset E$ satisfying
    $E/E'\cong k_x^{\oplus r}$. For each such
    subbundle, there is exactly one compatible $H^x$-level structure $\phi'_x$. Moreover, we have a decomposition
    $$
    K_x\Delta_x K_x=\bigsqcup_{C\in M_{r\times (n-r)}(k_x)}\mat{1}{C}{0}{1}\Delta_x K_x
    $$
    parametrizing these subbundles by $M_{r\times (n-r)}(k_x)$.
    \item Assume that $U^-\subset H^x\subset P^-$ or $H^x=1$. Then there is a unique subbundle
    $E'\subset E$ with $E/E'\cong k_x^{\oplus r}$ for which there exists an edge from $(\cE,\phi)$. Moreover, we have a decomposition
    $$
    K_x\Delta_x K_x=\bigsqcup_{C\in M_{r\times (n-r)}(k_x)}\mat{1}{\pi_x^{d_x}C}{0}{1}\Delta_x K_x
    $$
    parametrizing these edges $M_{r\times (n-r)}(k_x)$.
\end{enumerate}
In both cases, the level structures at the divisor
$D':=D-d_x[x]$ transport uniquely via the induced isomorphism on
$X\setminus \{x\}$.
\end{theorem}

\begin{proof}
    By Proposition~\ref{p:Hecke_correspondence_for_vector_bundles}, exactly as in the proof of \cite[Theorem 3.3]{KZ26}, a target of the Hecke
correspondence $\Phi_x^{\omega_r}$ is determined by a lattice
\[
N=g\cO_x^n\subset \cO_x^n,\qquad g\in K_x\Delta_xK_x,
\]
and the level structures on $D':=D-d_x[x]$ transport uniquely, since the
Hecke modification is an isomorphism on $X\setminus\{x\}$.

Write $g=k_1\Delta_xk_2$ with $k_1,k_2\in K_x$. Since
$k_2\in GL_n(\cO_x)$ stabilizes $\cO_x^n$, the underlying subbundle depends only on the
lattice
\[
N=k_1\Delta_x\cO_x^n.
\]
Two elements $k_1,k_1'\in K_x$ determine the same lattice if and only if
\[
k_1\Delta_x\cO_x^n=k_1'\Delta_x\cO_x^n,
\]
equivalently,
\[
k_1^{-1}k_1'\in
K_x\cap \Delta_x\GL_n(\cO_x)\Delta_x^{-1}.
\]
Hence the set of underlying subbundles is identified with
\[
K_x\Big/\Bigl(K_x\cap
\Delta_x\GL_n(\cO_x)\Delta_x^{-1}\Bigr).
\]

Write a matrix $h\in\GL_n(\cO_x)$ in block form
\[
h=
\begin{pmatrix}
X_{11}&X_{12}\\
X_{21}&X_{22}
\end{pmatrix},
\]
where $X_{11}$ is of size $r\times r$. Then
\begin{equation}\label{eq:conjugation_by_Delta}
    \Delta_xh\Delta_x^{-1}=
\begin{pmatrix}
X_{11}&\pi_x^{-1}X_{12}\\
\pi_xX_{21}&X_{22}
\end{pmatrix}.
\end{equation}
Now, we switch to case study.

{\bfseries Case $U\subset H^x\subset P$.} Since $K_x$ is the preimage of a subgroup contained in the upper parabolic $P(\cO_{d_x[x]})$,  the lower-left block of any  $h\in K_x$ satisfies
$X_{21}\in\pi_x^{d_x}M_{(n-r)\times r}(\cO_x)$; in particular $X_{21}\in\pi_xM_{(n-r)\times r}(\cO_x)$. We then see that
\begin{align*}
    K_x\cap
\Delta_x K_x\Delta_x^{-1}=K_x\cap\Delta_x \GL_n(\cO_x)\Delta_x^{-1}=\left\{\mat{X_{11}}{X_{12}}{X_{21}}{X_{22}}\in K_x:X_{12}\in\pi_xM_{r\times(n-r)}(\cO_x)\right\}.
\end{align*}
Since the set of neighbors of $g$ under the Hecke correspondence is identified with
\[
K_x\Big/\Bigl(K_x\cap
\Delta_x K_x\Delta_x^{-1}\Bigr),
\]
this shows that all bundles underlying different Hecke correspondences are pairwise distinct. 

To show that the edges are parametrized by matrices in $M_{r\times(n-r)}(k_x)$, use matrices of the form $\smat{1}{C}{0}{1}$ for all $C\in M_{r\times(n-r)}(k_x)$. For $h\in K_x\cap
\Delta_x K_x\Delta_x^{-1}$ as above, we get
$$
\mat{1}{C}{0}{1}h=\mat{X_{11}+CX_{21}}{X_{12}+CX_{22}}{X_{21}}{X_{22}}.
$$
From this, knowing that $X_{12}\in\pi_xM_{r\times(n-r)}(\cO_x)$ and $X_{21}\in \pi_x^{d_x} M_{(n-r)\times r}(\cO_x)$, we easily see that the coset
$$
\mat{1}{C}{0}{1}(K_x\cap
\Delta_x K_x\Delta_x^{-1})
$$
consists of all matrices $(X_{ij})\in K_x$ with 
$
X_{12}\equiv CX_{22}\pmod{\pi_x}.
$
Since $X_{22}$ is invertible for all matrices in $K_x$, these cosets are pairwise distinct and cover the entire $K_x$. This finishes the proof of $(i)$.

{\bfseries Case $U^-\subset H^x\subset P^-$ or $H^x=1$.} The case $H^x=1$ was covered in \cite[Theorem 3.3]{KZ26}, so we focus on $U^-\subset H^x\subset P^-$. Since $X_{12}\in\pi_xM_{r\times(n-r)}(\cO_x)$ for $h\in K_x$, (\ref{eq:conjugation_by_Delta}) implies that $K_x\subset\Delta_x \GL_n(\cO_x)\Delta_x^{-1}$. This implies the first claim of $(ii)$.

To get the second claim, note that the same equation gives
\begin{align}\label{eq:K_cap_Delta_conjugate_for_H_in_P_}
    K_x\cap
\Delta_x K_x\Delta_x^{-1}=\left\{\mat{X_{11}}{X_{12}}{X_{21}}{X_{22}}\in K_x:X_{12}\in\pi_x^{d_x+1}M_{r\times(n-r)}(\cO_x)\right\}.
\end{align}
Using similar arguments as above, we see that the cosets
$$
\mat{1}{\pi_x^{d_x}C}{0}{1}(K_x\cap
\Delta_x K_x\Delta_x^{-1}),\qquad C\in M_{r\times(n-r)}(k_x)
$$
are all distinct and cover the entire $K_x$. This finishes the proof.
\end{proof}
\begin{remark}\label{r:strengthening_of_general_connections}
    {\rm
        The conditions on $H^x$ in Theorem \ref{t:hecke_gln_cases}(ii) can be relaxed in the following way: $H^x\subset P^-$ and 
        $$
        \begin{pmatrix}
X_{11}&0\\
X_{21}&X_{22}
\end{pmatrix}\in H^x(\cO_{d_x[x]})\Rightarrow \begin{pmatrix}
X_{11}&0\\
\pi_xX_{21}&X_{22}
\end{pmatrix}\in H^x(\cO_{d_x[x]}).
        $$
        This property was implicitly used in (\ref{eq:K_cap_Delta_conjugate_for_H_in_P_}). The proof goes unchanged. This generalization includes the previous cases and any subgroup of the Levi subgroup $M$ of $P$.
    }
\end{remark}    

We finish this section with a description of the way to transfer from $G=\GL_n$ to $\bar G=\PGL_n$. Let $Z$ be the center of $G$. As before, assume $\cH$-ramification at a divisor $D=D'+d[x]$ and denote $K:=K_\cH(D)$. We need the following assumption on the subgroup $H^x$ at $x$: either $Z\subset H^x$, or there exists $i$ such that the $i$-th diagonal entry of each element of $H^x(\cO_{d[x]})$ equals $1$ modulo $\pi_x^d$. For example, such condition is satisfied by any parabolic subgroup containing $T$ and its unipotent subgroups.
\begin{proposition}\label{p:from_GL_to_PGL}
With the above ramification assumptions on $H^x$, let $\Gamma_G$ and $\Gamma_{\bar G}$ be the Hecke graphs for $G$ and $\bar G$ associated with $\Phi_x^{\omega_r}$. Then
\begin{enumerate}[label=(\roman*)]
    \item The quotient of the set of vertices of $\Gamma_G$ by the action of the abelian group $Z(F)\backslash Z(\AA)/(K\cap Z(\AA))$ is precisely $\bar G(F)\backslash \bar G(\AA)/\bar K=\Gamma_{\bar G}$. Geometrically, vertices of $\Gamma_{\bar G}$ are vector bundles with a level structure at $D$ up to tensoring with a line bundle with a level structure at $D$ (given by subgroups $H^y\cap Z$ for $y\in\supp D$).
    \item For $a,b\in \Gamma_{\bar G}$, fix a lift $\tilde a$ of $a$ to $\Gamma_G$. Then the projection between $\Gamma_G$ and $\Gamma_{\bar G}$ gives a bijection between the set of edges $a\to b$ in $\Gamma_{\bar G}$ and the set of edges between $a$ and all possible lifts of $b$ to $\Gamma_G$.
\end{enumerate}
\end{proposition}

\begin{proof}
The first item is clear. To prove $(ii)$, recall that after decomposing $K\Delta K$ into right cosets $\bigsqcup_i\tau_i\Delta K$, we get that all neighbors of $g\in \Gamma_G$ are $G(F)g\tau_i K$, counted with multiplicities. Same applies to $\bar G$. Therefore, item $(ii)$ is equivalent to the map
$$
K\Delta K/K\to K\Delta KZ(\AA)/KZ(\AA)
$$
being bijective. The $K$-action by left multiplication on both sets is transitive, and the stabilizers of $\Delta$ are $\Delta K\Delta^{-1}\cap K$ and $\Delta K\Delta^{-1}Z(\AA)\cap K$, respectively. Thus, we need to prove that
$$
\Delta K\Delta^{-1} Z(\AA)\cap K=\Delta K\Delta^{-1}\cap K.
$$

Since $Z(\AA)\subset \prod_{y\in |X|}Z(F_y)$, it is sufficient to prove that
$$
\Delta_y K_y\Delta_y^{-1} Z(F_y)\cap K_y=\Delta_y K_y\Delta_y^{-1}\cap K_y,\qquad \forall_{y\in |X|}.
$$
For $y\ne x$, this is evident since $\Delta_y=\mathrm{id}$ in this case. Thus, we are left with the case $y=x$.

Take $\gamma \in \Delta_x K_x\Delta_x^{-1} Z(F_x)\cap K_x$. Write
$$
    \gamma = \Delta_x p \Delta_x^{-1} z,\qquad p\in K_x,z\in Z(F_x).
$$
It is enough to prove that $z \in K_x$, which means that $z\mod\pi_x^{d_x}\in Z(\cO_{d_x[x]})\cap H^x(\cO_{d_x[x]})$.

Since $p,\gamma\in \GL_n(\cO_x)$, we have
\[
\det(p),\det(\gamma)\in \cO_x^\times.
\]
On the other hand,
\[
\det(\gamma)
=
\det(\Delta_x p\Delta_x^{-1})\zeta^n
=
\det(p)\zeta^n.
\]
Hence $\zeta^n\in \cO_x^\times$, so $\zeta\in\cO_x^\times$ and therefore $z\in Z(\cO_x)$.

If $Z\subset H^x$, then $Z(\cO_{x})\cap K_x=Z(\cO_x)$ and we are done. Assume now that there is  an index $i$  such that the $i$-th diagonal entry of  $H^x(\cO_{d[x]})$ is always  $1$. Since $p,\gamma\in K_x$, their reductions modulo $\pi_x^{d_x}$ lie in $H^x$. Conjugation by $\Delta_x$ does not change diagonal entries, so the $i$-th diagonal entries of $p$ and $\gamma=\Delta_xp\Delta_x^{-1}z$ are respectively $p_{ii}$ and $p_{ii}\zeta$. Therefore
\[
p_{ii},\,p_{ii}\zeta\equiv 1
\pmod{\pi_x^{d_x}}.
\]

It follows that $\zeta\equiv 1\pmod{\pi_x^{d_x}}$, and therefore
\[
z\bmod \pi_x^{d_x}=\mathrm{id}\in Z(\cO_{d_x[x]})\cap H^x.
\]
 Thus, $z\in K_x$, so we are done.

\end{proof}

\subsection{Hecke correspondences for general $G$}\label{Hecke Corres.G} We use the notation from Section \ref{s:preliminaries_on_Hecke_graphs}. We describe the Hecke operator $\Phi_{D, x}^\mu$ in adelic and geometric languages. Since the proofs are similar to the $\GL_n$ case, we will omit them.

\begin{proposition}\label{t:general_hecke_correspondence}
Fix a vertex $(\mathcal{P}, \psi) \in \Bun_{G,D}(k)$. The directed edges $(\mathcal{P}, \psi) \rightarrow (\mathcal{P}', \psi')$ in the Hecke graph $\Gamma_{D, \mu, x}$ of $\Phi_{D, x}^\mu$ are in bijection with the set of equivalence classes of pairs $(\mathcal{P}', \beta)$, where $\mathcal{P}'$ is a principal $G$-bundle on $X$ and $\beta: \mathcal{P}|_{X \setminus x} \xrightarrow{\sim} \mathcal{P}'|_{X \setminus x}$ is an isomorphism, such that:
\begin{enumerate}[label=(\roman*)]
    \item For any trivializations $\tilde{\psi}_x, \tilde{\psi}'_x$ of $\mathcal{P}|_{\mathcal{O}_x},\mathcal{P}'|_{\mathcal{O}_x}$ lifting $\psi_x,\psi_x'$, respectively, the local transition element $\tilde{\psi}_x \circ \beta^{-1} \circ (\tilde{\psi}'_x)^{-1} \in G(F_x)$ lies in the double coset $K_x \Delta_x^\mu K_x$.
    \item The level structure is strictly preserved away from the modification point, meaning $\psi'_y \circ \beta|_{D'} = \psi_y$ for any $y\in\operatorname{supp}D'$.
\end{enumerate}
\end{proposition}

For computations, the following statement will be useful. Let \(\mathcal H_1=\{H_1^y\}\) and \(\mathcal H_2=\{H_2^y\}\) be two ramification data over $D$ with
\[
H_1^y\subset H_2^y\subset G(\cO_{d_y[y]}),
\qquad \text{for all }y\in\operatorname{supp}D,
\]
and let \(K^1,K^2\) be the corresponding compact open subgroups of $G(\AA)$. Let
\(\Gamma_D^{\mathcal H_1}\) and \(\Gamma_D^{\mathcal H_2}\) denote the corresponding graphs. The above assumption gives a surjective map of sets 
    \begin{equation}\label{eq:refining_level_structures}
        \Gamma^{\cH_1}_{D}\to \Gamma^{\cH_2}_{D}
    \end{equation}
    given by extending the $H_1^y$-reductions to $H_2^y$ for every $y\in\supp D$.
\begin{proposition}\label{p:change_of_ramification}
    Assume that the map
    \begin{equation}\label{eq:change_of_ramification}
        K^1_x\Delta_x K^1_x/K^1_x\to K^2_x\Delta_x K^2_x/K^2_x
    \end{equation}
    induced by the inclusion is a bijection. Then for any $a,b\in \Gamma^{H^2}_{D}$ and a lift $\tilde a\in \Gamma^{H^1}_{D}$ of $a$ the map (\ref{eq:refining_level_structures})
    induces a bijection between the set of edges $a\to b$ and the set of edges from $\tilde a$ to all possible lifts of $b$ to $\Gamma^{\cH_1}_{D}$.
\end{proposition}
\begin{proof}
    The assumption is equivalent to the following: if we write
    $$
    K^1_x\Delta_x K^1_x=\bigsqcup_i \tau_i K_x^1
    $$
    for some $\tau_i\in K^2_x$, then
    $$
    K^2_x\Delta_x K^2_x=\bigsqcup_i \tau_i K_x^2.
    $$
    Therefore, in both graphs, an element $g\in G(\AA)$ is connected to $\tau_ig$ for all $i$, where $\tau_i$ are considered lifted to elements in $G(\AA)$ such that $\tau_{i,y}=\mathrm{id}$ for all $y\in |X|\setminus x$. This clearly induces the bijection stated.
\end{proof}
\begin{remark}\label{r:change_of_ramification}
    {\rm
     Proposition \ref{p:change_of_ramification} is a powerful tool for getting new examples from old ones. Notice that the assumption depends on the chosen ramification at $x$. Therefore, if $H^x_1=H^x_2\subset G$ and $H_1^y\subset H_2^y\subset G$ for $y\in\supp D\setminus x$, the graph for $\cH_2$ is obtained by simply identifying vertices in the graph for $\cH_1$ and keeping the same set of edges. In particular, if one computes the example with $H^y=\mathrm{id}$ for all $y\ne x$, then one automatically gets the example for any choice of $H^y$, $y\ne x$ with the same $H^x$.
    }
\end{remark}

\newpage

\section{Hecke graphs at the cusp} \label{s:general_G}
\subsection{Notation and preliminaries}
As usual, we use notation of Section \ref{s:preliminaries_on_Hecke_graphs}. Additionally, we introduce the following notation,
compatible with \cite{Schieder2015} and \cite{KZ26}:
\begin{itemize}
    \item $T$ is a fixed split maximal torus of $G$, and $B$ is a fixed Borel subgroup of $G$
    containing $T$.
    \item $U$ is the unipotent radical of $B$.
    \item $\Phi=\Phi(G,T)$ is the root system of $G$ with respect to $T$, $W$ is the Weyl group of $G$.
    \item $\Phi^+=\Phi^+(G,T,B)$ and $\Phi^-=-\Phi^+$ are the sets of positive and negative roots.
    \item $\alpha_1,\ldots,\alpha_n$ and $\check\alpha_1,\ldots,\check\alpha_n$ are the simple roots
    and simple coroots, respectively.
    \item For a root $\alpha\in \Phi$, $U_\alpha$ is the corresponding root subgroup of $G$.
    \item $\check\Lambda_G:=X_*(T)$ and $\Lambda_G:=X^*(T)$ are the coweight and weight lattices,
    respectively.
    \item For $\chi\in X^*(T)$, we define $k_{\chi}$ to be the corresponding $1$-dimensional representation of $T$.
\end{itemize}

For a root $\alpha\in \Phi$ and a coweight $\nu\in X_*(T)$, we write
\[
\langle \alpha,\nu\rangle\in \mathbb Z
\]
for the natural pairing.

Let $P$ be a parabolic subgroup of $G$ containing $B$. We define the following notions:
\begin{itemize}
    \item $I_P\subset \{1,\ldots,n\}$
    is the set of $i$ such that
    $
    \fg_{-\alpha_i}\subset \operatorname{Lie}(P).
    $
    This set characterizes $P$ uniquely.
    \item $U(P)$ is the unipotent radical of $P$, and
    $
    M:=P/U(P)
    $
    is the Levi quotient.
    \item We write $\check\Lambda_M$ and $\Lambda_M$ for the coweight and weight lattices of $M$.
     \item For a character $\chi:P\to \mathbb G_m$, we also write $k_\chi$ for the corresponding
    $1$-dimensional representation of $P$.
    \item A character $\chi:P\to \mathbb G_m$ is called {\bfseries dominant} if its restriction to $T$
    is a nonnegative integral combination of the simple roots $\alpha_i$ with $i\notin I_P$.
    \item We write
    \[
    \Lambda_{G,P}:=\{\chi:P\to \mathbb G_m \mid \chi|_{Z(G)^\circ}=1\},
    \qquad
    \check\Lambda_{G,P}:=\operatorname{Hom}_{\mathbb Z}(\Lambda_{G,P},\mathbb Z).
    \]
    \item If $\mathcal P_P$ is a principal $P$-bundle on $X$, its degree is the element
    \[
    \deg_P(\mathcal P_P)\in \check\Lambda_{G,P}
    \]
    characterized by
    \[
    \langle \chi,\deg_P(\mathcal P_P)\rangle
    :=
    \deg\bigl(\mathcal P_P\times^P k_\chi\bigr)
    \qquad\text{for all }\chi\in \Lambda_{G,P}.
    \]
    \item An element $\check\lambda_P\in \check\Lambda_{G,P}$ is called {\bfseries dominant} if
    \[
    \langle \chi,\check\lambda_P\rangle\ge 0
    \]
    for every dominant character $\chi\in \Lambda_{G,P}$, and it is called
    {\bfseries dominant $P$-regular} if
    \[
    \langle \chi,\check\lambda_P\rangle>0
    \]
    for every nontrivial dominant character $\chi\in \Lambda_{G,P}$.
\end{itemize}

Fix a decomposition $D=D_1+D_2$ into a sum of effective $D_1$ and $D_2$ satisfying
\begin{equation}\label{eq:d1_d2_condition}
    \operatorname{supp}(D_1)\cap\operatorname{supp}(D_2)=\varnothing,
\qquad
x\notin\operatorname{supp}(D_1).
\end{equation}
We have a map
$$
p_{D,D_2}:\Bun_{D}^\cH(X)(k)\to \Bun_{D_2}^\cH(X)(k)
$$
forgetting the level structure outside $D_2$.

\begin{definition}
A principal $G$-bundle $\cP_G$ on $X$ is called {\bfseries semi-stable} if for every parabolic
subgroup $P\subset G$, every reduction $\cP_P$ of $\cP_G$ to $P$, and every dominant character
$\chi:P\to \mathbb G_m$ whose restriction to $Z(G)^\circ$ is trivial, one has
\[
\deg\bigl(\cP_P\times^P k_\chi\bigr)\le 0.
\]
\end{definition}

\begin{theorem} \label{HN reduction}
Let $\mathcal P_G$ be a principal $G$-bundle on $X$. Then there exists a unique parabolic
subgroup $P\subset G$, a unique dominant $P$-regular element $\check\lambda_P\in \check\Lambda_{G,P}$, and a unique reduction $\mathcal P_P$ of $\mathcal P_G$ to $P$ such that 
\[
\mathcal P_P\in\Bun^{ss}_{P,\check\lambda}
:=
\left\{
\mathcal P_P \in \Bun_P
\;\middle|\;
\deg(\mathcal P_P)=\check\lambda
\text{ and }
\mathcal P_P/U(P)\text{ is semi-stable}
\right\}.
\]

Equivalently, if $M:=P/U(P)$ is the Levi quotient, then the induced $M$-bundle
$\mathcal P_M:=\mathcal P_P/U(P)$
is semi-stable, and the degree of the reduction $\mathcal P_P$ is dominant $P$-regular.
\end{theorem}

\begin{proof}
Follows from \cite[Theorem 2.3.3(a),(c) and Remark 2.4.1]{Schieder2015}.

\end{proof}

In the situation of the above theorem, we call $\cP_P$ the {\bfseries Harder-Narasimhan (HN) reduction of $\cP_G$}.

 We say that {\bfseries $\cP_G$ has HN reduction to $B$},
if the canonical Harder--Narasimhan reduction of $\cP_G$ is a reduction $\cP_B$ to the Borel subgroup $B$. In that case, we write
\[
\cP_T:=\cP_B/U
\]
for the induced $T$-bundle, and for each $\alpha\in \Phi^+$ we set
\[
\cL_\alpha:=\cP_T\times^T k_\alpha,
\]
where $k_\alpha$ is the one-dimensional $T$-representation of weight $\alpha$.

We will be interested in the cusp locus of $\Bun_G$.

\begin{definition}[The $D$-cusp locus]\label{def:Dcusp}
A $G$-bundle $\cP$ on $X$ is {\bfseries in the $D$-cusp} if it has HN reduction to $B$ and satisfies
\[
\deg(\cL_\alpha)>2g-2+\deg(D)\qquad\text{for all }\alpha\in \Phi^+.
\]
The loci of such bundles (with level structures) in $\mathrm{Bun}_G$ and $\mathrm{Bun}_{G,D}$ are denoted by 
$\mathrm{Bun}_{G}^{D\text{-cusp}}$ and $\mathrm{Bun}_{G,D}^{D\text{-cusp}}$, respectively.
\end{definition}

As in Section \ref{s:preliminaries_on_Hecke_graphs}, fix a dominant coweight $\mu\in X_*(T)$.

\begin{definition}[Cusp locus]\label{def:deepcusp}
A $G$-bundle $\cP\in \mathrm{Bun}_{G}(k)$ is {\bfseries in the $(D,x,\mu)$-cusp}
if $\cP$ has HN reduction to $B$ and
\[
\deg(\cL_\alpha)>2g-2+\deg(D)+\langle \alpha,w(\mu)\rangle\,\deg(x)
\qquad\text{for all }\alpha\in \Phi^+, w\in W.
\]
The loci of such bundles (with level structures) in $\mathrm{Bun}_G$ and $\mathrm{Bun}_{G,D}$ are denoted respectively by 
$\mathrm{Bun}_{G}^{D\text{-cusp}}(D,x,\mu)$ and $\mathrm{Bun}_{G,D}^{D\text{-cusp}}(D,x,\mu)$.
\end{definition}

We use the shorthand
\[
G(\mathcal O_D),\qquad B(\mathcal O_D),\qquad U(\mathcal O_D),\qquad T(\mathcal O_D)
\]
for the $\mathcal O_D$-points. We regard $T(k)\subset T(\mathcal O_D)$ via constant sections.

\subsection{Fibers of the structure forgetting map}

For the next results, it is convenient to define the following object:
\begin{definition}\label{def:huge_torus_intersection}
    Let $\cH$ be a ramification datum over $D$ and $\tau=(\tau_y)_{y\in\supp D}\in G(\cO_D)$. We set
    $$
    T_D(\tau):=T(k) \cap \bigcap_{y\in\supp D}\tau_y H^y(\mathcal O_{d_y[y]})\tau_y^{-1}\,U(\mathcal O_{d_y[y]}).
    $$
\end{definition}

\begin{theorem}\label{t:automorphisms_reductive_general}
Let $\mathcal{P}_G$ be a principal $G$-bundle on $X$ which is $D$-cusp (Definition~\ref{def:Dcusp}), and $D=D_1+D_2$ as in (\ref{eq:d1_d2_condition}). Fix a ramification datum $\cH$ over $D_2$ and let $\phi_{D_2}=(\cP_{H^y}|_{d_y[y]})_{y\in\supp D_2}$ denote an $\cH$-level structure over $D_2$. Then:
\begin{enumerate}[label=(\roman*)]
    \item The canonical Harder-Narasimhan $B$-reduction $\mathcal{P}_B$ is induced from $\mathcal{P}_T := \mathcal{P}_B/U$, i.e.,
    \[
    \mathcal{P}_B \simeq \mathcal{P}_T \times^T B.
    \]
    \item Fix a $\cP_B$-compatible trivialization $\psi: \mathcal{P}_G|_{D} \simeq G \times \operatorname{Spec}(\mathcal{O}_{D})$ in the sense of Definition \ref{def:Bcompatible} and assume that the image of the level structure $\cP_{H^y}|_{d_y[y]}$ at $y\in\supp D_2$ is 
    $$
    \tau_y\cdot(H^y\times \Spec\cO_{d_y[y]}),\qquad \text{for some }\tau=(\tau_y)_{y\in\supp D_2}\in G(\cO_{D_2}).
    $$  
    Under the induced identification $\operatorname{Aut}(\mathcal{P}_G|_{D_1}) \cong G(\mathcal{O}_{D_1})$, the image of the restriction homomorphism
    \[
    \mathrm{res}_{D_1}: \operatorname{Aut}(\mathcal{P}_G, \phi_{D_2}) \longrightarrow \operatorname{Aut}(\mathcal{P}_G|_{D_1}) \cong G(\mathcal{O}_{D_1})
    \]
    coincides with the subgroup of $B(\mathcal{O}_{D_1}) \subset G(\mathcal{O}_{D_1})$ given by 
    \[
    \begin{cases}
        T(k)\ltimes U(\cO_{D_1}), &D_2=0,\\
        T_{D_2}(\tau)\ltimes U(\cO_{D_1}),&D_2\ne 0.
    \end{cases}
    \]
\end{enumerate}
\end{theorem}
\begin{proof}

    \textit{(i).} The proof that $\mathcal{P}_B \simeq \mathcal{P}_T \times^T B$ relies entirely on the fact that $\mathcal{P}_G$ is $D$-cusp. Because this condition forces $\deg(\mathcal{L}_\alpha) > 2g - 2$ for all positive roots $\alpha$, the non-abelian cohomology group $H^1_{\mathrm{fppf}}(X, U_{\mathcal{P}_T})$ is trivial. This geometric derivation is identical to the principal ramification case; see the exact height filtration argument in \cite[Theorem 5.8(i)]{KZ26}.\\

    \textit{(ii).} 
We know from \cite[Theorem 5.8(ii)]{KZ26} that
\[
\operatorname{Im}\bigl(\Aut(\mathcal P_G)\to \Aut(\mathcal P_G|_D)\simeq G(\mathcal O_D)\bigr)
=
T(k)\ltimes U(\mathcal O_D),
\]
which, in particular, proves the $D_2=0$ case. From now, we will assume that $D_2\ne 0$.

We first analyze the image $A$ of $\Aut(\cP_G,\phi)$ in $G(\cO_D)$. Using \(D=D_1\sqcup D_2\), we write every element in  this image in a form
\[
(tu_1,tu_2),
\qquad t\in T(k),\quad u_1\in U(\mathcal O_{D_1}),\quad u_2\in U(\cO_{D_2}).
\]\
The component \(tu_2\) must fix the subbundle \(\tau_y H^y\), which happens if and only if
\[
(tu_2)\tau_y H^y(\cO_{d_y[y]})=\tau_y H^y(\cO_{d_y[y]}),
\]
equivalently,
\[
t\in \tau_y H^y(\cO_{d_y[y]})\tau_y^{-1}u_2^{-1}.
\]

Intersecting over all $y\in\supp D_2$, we obtain that $t\in T_{D_2}(\tau)$. Since the $\im\mathrm{res}_{D_1}$ equals the image of $A$ under the projection map
$$
G(\cO_D)\to G(\cO_{D_1}),
$$
we have
$$
\im\mathrm{res}_{D_1}\subset T_{D_2}(\tau)\ltimes U(\cO_{D_1}).
$$
On the other hand, take an element $tu_1\in T_{D_2}(\tau)\ltimes U(\cO_{D_1})$. Running the argument backwards, we find an element $u_2\in U(\cO_{D_2})$ such that
$
(tu_1,tu_2)\in A.
$
Then $tu_1\in \im\mathrm{res}_{D_1}$, which finishes the proof.
\end{proof}

\begin{corollary}\label{cor:fiber_size_reductive}
With the notation and assumptions of Theorem~\ref{t:automorphisms_reductive_general}, there is a bijection
\begin{align*}
p_{D,D_2}^{-1}(\mathcal P_G,\phi_{D_2})
&\cong
\cH(\cO_{D_1})\backslash G(\cO_{D_1})/T_{D_2}(\tau)U(\cO_{D_1})
\\
&\cong \left(
\cH(\cO_{D_1})\backslash G(\cO_{D_1})/U(\cO_{D_1})
\right)/T_{D_2}(\tau)
\end{align*}

\end{corollary}
\begin{proof}
    The $\cH$-level structures on $\cP_G$ over $D$ restricting to $\phi_{D_2}$ over $D_2$ are parametrized by
    $$
    \cH(\cO_{D_1})\backslash G(\cO_{D_1})/\Aut(\cP_G,\phi_{D_2}).
    $$
By Theorem~\ref{t:automorphisms_reductive_general}, $\Aut(\cP_G,\phi_{D_2})$ acts by right multiplication by its image which equals
\[
T_{D_2}(\tau)\ltimes U(\mathcal O_{D_1}).
\]
This finishes the proof.
\end{proof}

What we just proved shows that the fibers of the map $p_{D,D_2}$ can have different cardinalities, which never happened for the case $H^y=1$ studied in \cite{KZ26}. We introduce the following definition aimed to fix this peculiarity.

\begin{definition}\label{def:regular_datum}
    We call the ramification datum $(D,\cH)$ {\bfseries regular} if the subgroup $T_D(\tau)$ does not depend on the choice of $\tau\in G(\cO_D)$. In this case, we denote $T_{D}(\tau)$ simply by $T_D$. If we want to emphasize the divisor $D$, we call such a ramification datum {\bfseries regular at $D$}.
\end{definition}
\begin{remark}
    {\rm
        We note that $T_D(\tau)$ depends only on the class
        $$
        \tau_y\in B(\cO_{d_y[y]})\backslash G(\cO_{d_y[y]})/H^y(\cO_{d_y[y]}).
        $$
        Indeed, the independence on the right multiplication by $H^y$ is clear. For the left multiplication by $B$, choose $tu\in B(\cO_{d_y[y]})$ with $t\in T(\cO_{d_y[y]})$ and $u\in U(\cO_{d_y[y]})$. Since $T$ normalizes $U$, we have
        \begin{align*}
            T(k)\cap tu\tau H^y \tau^{-1} u^{-1}t^{-1}U&=T(k)\cap tu\tau H^y \tau^{-1} u^{-1}Ut^{-1}\\
            &=T(k)\cap u\tau H^y \tau^{-1} U.
        \end{align*}
        Take an element $t\in T(k)\cap u\tau H^y \tau^{-1} U$. Then 
        $$
        u^{-1}t\in \tau H^y \tau^{-1} U.
        $$
        Since $T$ normalizes $U$, there exists $u'\in U(\cO_{d_y[y]})$ such that $u^{-1}t=tu'$. Then 
        $$
        tu'\in \tau H^y \tau^{-1} U,
        $$
        which implies that $t\in \tau H^y \tau^{-1} U$. This shows the inclusion
        $$
    T(k)\cap tu\tau H^y \tau^{-1} u^{-1}t^{-1}U\subset T(k)\cap \tau H^y \tau^{-1} U,
        $$
        and the reverse inclusion is proven by reverting the argument.
    }
\end{remark}

The next proposition shows that any ramification with at least one unipotent subgroup is regular with $T_D=1$.
 \begin{proposition}\label{p:unipotent_implies_regular}
     If at least one $H^y$ is unipotent, then the corresponding ramification datum is regular.
 \end{proposition}
\begin{proof}
    Choose $y$ such that $H^y$ is unipotent. Note that reduction modulo $\pi_y$ or deleting points from $D$ can only increase $T_D(\tau)$, so it is enough to check the statement for the divisor $[y]$. Take $t\in T_D(\tau)$. Write
    $$
    t=hu,\qquad h\in \tau H^y(k_y)\tau^{-1}, u\in U(k_y).
    $$
    Then $h=u^{-1}t\in B(k_y)$. Since $h$ is unipotent and the only unipotent elements in $B(k_y)$ are in $U(k_y)$, we must have $h\in U(k_y)$, which implies that $t\in U(k_y)$. Since $t\in T(k)$, we conclude that $t=1$, as desired.
\end{proof}

\begin{definition}\label{def:covering_map}
    Let $\Gamma_1$ and $\Gamma_2$ be two oriented graphs such that every vertex has finite ingoing and outgoing degree. A {\bfseries covering map} $p:\Gamma_1\to\Gamma_2$ of graphs is a map on vertices such that
    \begin{itemize}
        \item $p$ is surjective on vertices.
        \item For each $v\in\Gamma_1$ and $w\in\Gamma_2$, $p$ induces a bijection between edges $p(v)\to w$ and $v\to p^{-1}(w)$, where we take the union of edges over all preimages of $w$.
        \item For each $v\in\Gamma_1$ and $w\in\Gamma_2$, $p$ induces a bijection between edges $p^{-1}(w)\to v$ and $w\to v$, where we take the union of edges over all preimages of $w$.
    \end{itemize}
\end{definition}

If we think of oriented graphs as one-dimensional CW complexes, then a covering map is a topological covering map preserving the CW structure and orientations of one-dimensional cells. See Figure 1 in \cite{KZ26} for a visualization. In particular, such a map is completely determined by $\Gamma_2$, the set of vertices of $\Gamma_1$, and the way topological loops lift.

Fix a ramification datum $(D,\cH)$ and let $\Gamma_{D}:=\Gamma_{D,x}^{\mu}$ and $\Gamma_{D_2}:=\Gamma_{D_2,x}^{\mu}$. Let 
$$
\Gamma_{D}^{>k}:=\{(\cP_G,\phi)\in \Gamma_{D}^{\mathrm{cusp}}:\forall_{\alpha\in \Phi^+,w\in W}\,\deg\cL_\alpha>2g-2+k+\deg(D)+\langle \alpha,w(\mu)\rangle\,\deg(x)\},
$$
where $\Gamma_{D}^{\mathrm{cusp}}$ denotes the $G$-bundles whose HN reduction is to $B$. Recall that we fixed a decomposition $D=D_1+D_2$ satisfying (\ref{eq:d1_d2_condition}).
\begin{theorem}\label{t:covering_map_GL}
    Assume that the ramification datum at $D$ is regular over $D_2$ in the sense of Definition \ref{def:regular_datum}. Then the map $p_{D,D_2}:\Gamma_{D}^{>k}\to \Gamma_{D_2}^{>k}$ is a covering map of graphs for $k\ge \deg D$.
\end{theorem}
\begin{proof}
    Surjectivity on points is clear. Since $x\notin\supp D_1$, the assumption of Proposition \ref{p:change_of_ramification} is satisfied, and hence the second property of a covering map is satisfied by $p_{D,D_2}$.
    
    For the third property, we mimic the strategy of the proof of \cite[Lemma 4.15]{KZ26}. Take $v\in\Gamma_{D}^{>k}$ and $w\in\Gamma_{D_2}^{>k}$, and set $p:=p_{D,D_2}$, $v':=p(v)$.
    We claim that the number of edges from  $p^{-1}(w)$ to $v$ is at most the number of edges from $w$ to $v'$. Let us show how this claim proves the third property of a covering map. By the second property just proven, the total number of edges from $p^{-1}(w)$ to $p^{-1}(v')$ equals $|p^{-1}(w)|$ times the number of edges from $w$ to $v'$. On the other hand, by our claim, the number of edges from $p^{-1}(w)$ to $p^{-1}(v')$ is at most $|p^{-1}(v')|$ times the number of edges from $w$ to $v'$, with equality if and only if the third property is satisfied. But equality must be satisfied since $|p^{-1}(v')|=|p^{-1}(w)|$. So, we are done as soon as we prove the claim.

    Suppose the contrary: that there are more edges from $p^{-1}(w)$ to $v$ than from $w$ to $v'$. Choose representatives in each isomorphism class of level structures for $w=(\cP_G,\phi_{D_2})$, $v=(\cP_G',\phi_{D_2}',\phi_{D_1}')$, and $(\cP_G,\phi_{D_2},\psi_{D_1})$ for any element in $p^{-1}(w)$. By Proposition \ref{t:general_hecke_correspondence}, edges from each $(\cP_G,\phi_{D_2},\psi_{D_1})$ to $(\tilde\cP_G,\tilde\phi_{D_2},\tilde\psi_{D_1})$ are given by pairs $(f,\tau)$, where $f:\tilde \cP_G|_{X\setminus x}\simeq\cP_G|_{X\setminus x}$ is an isomorphism over $X\setminus x$ and $\tau$ is one of the chosen set of representatives $\tau_i$ of $K\Delta K=\sqcup \tau_i K$ ($\tau$ is redundant if $x\notin\supp D$). By our assumption, there exist at least two edges from $p^{-1}(w)$ to $v$ given by the same pair $(f,\tau)$. Write those edges as
    \begin{align*}
        (\cP_G,\phi_{D_2},\psi_{D_1}^1)&\to (\tilde \cP_G,\tilde\phi_{D_2},\psi_{D_1}^1\circ f),\\
        (\cP_G,\phi_{D_2},\psi_{D_1}^2)&\to (\tilde \cP_G,\tilde\phi_{D_2},\psi_{D_1}^2\circ f),
    \end{align*}
    where the first two elements in the triples on the right-hand side are the same since we use the same $(f,\tau)$ for both connections. By assumption, the level structures on the right represent the same vertex $v$, hence there exists an isomorphism
    $$
    \alpha: (\tilde \cP_G,\tilde\phi_{D_2},\psi_{D_1}^1\circ f)\simeq (\tilde \cP_G,\tilde\phi_{D_2},\psi_{D_1}^2\circ f).
    $$
    
    Choose $B$-compatible trivializations of $\cP_G$ and $\tilde \cP_G$ over $D$. Since both $f$ and $\alpha$ preserve the HN reductions $\cP_B$ and $\tilde \cP_B$, we get
    $$
    f|_{D_1}\in B(\cO_{D_1}),\quad \alpha|_{D_1}\in T(k)U(\cO_{D_1}). 
    $$
    Moreover, if $D_2\ne 0$, then $\alpha|_D\in T_{D_2}(\tau)U(\cO_{D_1})$ by Theorem \ref{t:automorphisms_reductive_general}(ii). Then
    $$
    (f^{-1}\alpha f)|_{D_1}\in
    \begin{cases}
        T(k)U(\cO_{D_1}),&D_2=0,\\
         T_{D_2}(\tau)U(\cO_{D_1}),&D_2\ne 0,
    \end{cases}
    $$
    under a $B$-compatible trivialization of $\cP_G|_{D_1}$. By the assumption that the ramification datum is regular over $D_2$ and Theorem \ref{t:automorphisms_reductive_general}(ii), this is precisely the subgroups of maps that lift to an automorphism of $(\cP_G,\phi_{D_2})$. Such a lift of $f^{-1}\alpha f$ gives an isomorphism between $(\cP_G,\phi_{D_2},\psi_{D_1}^1)$ and $(\cP_G,\phi_{D_2},\psi_{D_1}^2)$. This is a contradiction since we chose one representative from each isomorphism class. This finishes the proof.
\end{proof}

For our main result, we introduce a $T(k)$-action on level structures. Let 
$$
(\cP_G,\phi)=(\cP_G,\phi|_{D_1},\phi|_{D_2})
$$
be an element in the cusp of $\Bun_{G,D}^\cH(k)$, where $\phi$ is considered a trivialization of $\cP_G$ over $D$ as in Remark \ref{rem:coset_conventions}. Choose a $\cP_B$-compatible reference trivialization of $\cP_G$. By Corollary \ref{cor:fiber_size_reductive}, we can treat $(\phi|_{D_1},\phi|_{D_2})$ as an element 
$$
(g_1,g_2)\in \big[\cH(\cO_{D_1})\backslash G(\cO_{D_1})/U(\cO_{D_1})\times \cH(\cO_{D_2})\backslash G(\cO_{D_2})/U(\cO_{D_2})\big]/T_{D}(\tau).
$$

Using this interpretation, we define a $T(k)$-action on the cusp of $\Bun_{G,D}^\cH$ by 
\begin{equation}\label{eq:T-action}
    t.(\cP_G,\phi):=(\cP_G,\phi|_{D_1}\cdot t,\phi|_{D_2}),\quad t\in T(k).
\end{equation}

\begin{lemma}
    The action (\ref{eq:T-action}) is well-defined and does not depend on the chosen reference trivialization. Moreover, the stabilizer of $(\cP_G,\phi)$ equals $T_{D_2}(\tau)$.
\end{lemma}
\begin{proof}
    The action does not depend on the choice of representatives of the double cosets since $T(k)$ normalizes $T$ and $U$, which proves well-definiteness. We only need to check that this definition does not depend on the chosen reference trivialization.

Since $B$-compatible trivializations are related by elements of $B(\cO_D)$, choose $b\in B(\cO_D)$ by which we are going to change the chosen trivialization. Write 
$$
b=b_Tb_U,\qquad b_T\in T(\cO_D),\;b_U\in U(\cO_D).
$$
Then $(g_1,g_2)$ and $(g_1t,g_2)$ transform to $(g_1b,g_2b)$ and $(g_1tb,g_2b)$, respectively. Since we quotient on the right by $U$, the last two pairs are equivalent to $(g_1b_T,g_2b_T)$ and $(g_1tb_T,g_2tb_T)$ in the double quotients. Commutativity $tb_T=b_Tt$ finishes the proof.

To compute the stabilizer, take $t\in T(k)$ such that $t.(\cP_G,\phi)=(\cP_G,\phi)$. This means that there exists an automorphism $\alpha$ of $\cP_G$ fixing $\phi_{D_1}$ such that $\phi|_{D_1}\circ \alpha =\phi|_{D_1}t$. By Theorem \ref{t:automorphisms_reductive_general}(ii), it exists if and only if its torus-part is $T_{D_2}(\tau)$. Since its torus-part is $t$, we are done.
\end{proof}

To state our main theorem, recall from Theorem \ref{t:covering_map_GL} and Corollary \ref{cor:fiber_size_reductive} that if the ramification datum $(D,\cH)$ is regular over $D_2$, then $p_{D,D_2}:\Gamma_{D}^{>k}\to \Gamma_{D_2}^{>k}$ is a covering map of graphs for $k\ge \deg D$ with fiber of size
    $$
    \big|\cH(\cO_{D_1})\backslash G(\cO_{D_1})/T_{D_2}U(\cO_{D_1})\big|.
    $$
    
Therefore, to understand $p_{D,D_2}$ in this case, we only need to understand how topological loops lift. For a point 
$$
v:=(\cP_G,\phi|_{D_1},\phi|_{D_2})\in \Gamma_{D}^{>k}
$$
and a loop $\gamma$ in $\Gamma_{D_2}^{>k}$ based in $p_{D,D_2}(v)=(\cP_G,\phi|_{D_2})$, denote the endpoint of the lift of this loop to $\Gamma_{D}^{>k}$ by $\gamma(v)$. We also denote by $D_2'$ the divisor $D_2-d_x[x]$, i.e. the divisor obtained from $D_2$ by removing the $x$-part.
\begin{theorem}\label{t:main_theorem}
    Assume that $(D,\cH)$ is regular over $D_2$ and $D_2'$. With the above notation, $\gamma(v)$ lies in the $T_{D_2'}/T_{D_2}$-orbit of $v$. In particular, if $T_{D_2}=T_{D_2'}$, then the covering is disjoint over each connected component.
\end{theorem}

\begin{proof}

As in the proof of \cite[Theorem 5.15]{KZ26}, the loop $\gamma$ gives $f\in\Aut(\cP_G|_{X\setminus x})$ preserving $\cP_B$ and $a\in\Aut(\cP_G)$ such that 
\begin{align*}
    \phi|_{D'}\circ f&=\psi|_{D'},\\
    \phi|_{D_2}\circ a&=\psi|_{D_2}.
\end{align*}
The map $f$ is composed by isomorphisms over $X\setminus x$ in the Hecke correspondences and $a$ exists by the assumption that $\gamma$ is a loop (the level structures over $D_2$ are isomorphic).

Fix a $B$-compatible reference trivialization of $\cP_G$ over $D$. Since both $f$ and $a$ preserve $\cP_B$, we have $f|_{D}\in T(k)U(\cO_{D})$ and $a\in T(k)U(\cO_D)$. Write $f=f_Tf_U$ and $a=a_Ta_U$ with respect to these decompositions and let $t:=a_T/f_T\in T(k)$. Since $af^{-1}$ preserves $\psi|_{D_2'}$, we must have $t\in T_{D_2'}$ by Theorem \ref{t:automorphisms_reductive_general}(ii). Moreover,
$$
(f|_{D_1}t,a|_{D_2})\in T(k)U(\cO_{D_1})\times T(k)U(\cO_{D_2})
$$
lifts to an automorphism of $\cP_G$ by the same Theorem. This automorphism intertwines $v=(\phi_{D_1},\phi_{D_2})$ with $t.\gamma(v)=(\psi_{D_1}t,\psi_{D_2})$, which proves that $\gamma(v)$ lies in the $T_{D_2'}$-orbit of $v$. This finishes the proof.

\end{proof}

\subsection{Summary and the general strategy for computations} \label{ss:summary_computations}
Let us summarize what we have done and how to compute the graphs of Hecke operators for general ramification. We claim that to know all the graphs at the cusp, it is enough to compute the following graphs at the cusp:
\begin{enumerate}[label=(\roman*)]
    \item $D=0$.
    \item $\supp D=\{x\}$, $H^x=1$.
    \item $\supp D=\{x,y\}$, $H^y=1$, any $H^x$ (but see Proposition \ref{p:change_of_ramification} for ways to map between different $H^x$).
\end{enumerate}
Note that if we compute $(iii)$ with $H^x=1$, then we automatically get $(ii)$ by forgetting the level structure at $y$, by Proposition \ref{p:change_of_ramification}.
 Take a ramification datum $(D,\cH)$. We will use Theorem \ref{t:main_theorem} for various decompositions $D=D_1+D_2$ as in (\ref{eq:d1_d2_condition}).
\begin{enumerate}
    \item If $x\notin\supp D$, then take $D_2=0$. The ramification datum is regular with $T_{D_2}=T_{D_2'}=T(k)$, hence the graph at the cusp is a number of copies of the unramified graph ($(i)$ above).
    \item If $\supp D=\{x\}$, then this is case $(ii)$ above.
    \item If $x\in\supp D$ and there exists $x\ne y\in\supp D$ with $H^y=1$, use $D_2$ with $\supp D_2=\{x,y\}$. This is the case $(iii)$ above. By Proposition \ref{p:unipotent_implies_regular}, $\cH$ is regular over $D_2$ and $D_2'$ with $T_{D_2}=T_{D_2}'$. Therefore, the desired graph at the cusp is a number of copies of the graph computed in $(iii)$ above.
    \item If $x\in\supp D$ and there exists $x\ne y\in\supp D$, start with the ramification datum $\cH'$ with $H'^z=H^z$ for $z\ne y$ and $H^y=1$ for $z=y$. This graph was computed in the previous case. By Proposition \ref{p:change_of_ramification}, the graph for $\cH$ at the cusp is obtained by identifying certain vertices in the graph for $\cH'$ without changing the edges.
 \end{enumerate}
 
If one wants to compute the entire graph, then Proposition \ref{p:change_of_ramification} is the main tool. See Remark \ref{r:change_of_ramification}.

\section{Computations for $\PGL_2$}
\subsection{Spaces of eigenforms for $\PGL_2$ unramified at $x$}\label{ss:unramified_PGL2}

Fix a point $x\in |X|$ of degree $r$ and a divisor $D$ on $X$ such that $x\notin\operatorname{supp}D$. This case was studied in \cite[Section 4.3]{KZ26}, and the results in our new setting are completely analogous. Let $\Gamma$ be the graph of the corresponding Hecke operator. Use the following decomposition:
\begin{itemize}
    \item $\Gamma_i:=\{\cL_1\oplus\cL_2:ri<\deg \cL_1-\deg\cL_2-(2g-2+\deg D)\le r(i+1)\}/\sim$,
    \item $\Gamma':=\Gamma\setminus\bigsqcup_{i\ge 1}\Gamma_i$,
\end{itemize}
where $\sim$ is the equivalence relation given by tensoring with a line bundle with a level structure at $D$. The graph is presented on Figure \ref{fig:PGL2_unramified}, where the number of possible level structures in the fibers equals
$$
\left|\cH(\cO_D)\setminus G(\cO_D)/T(k)\ltimes U(\cO_D)\right|.
$$

\begin{figure}
    \centering
\scalebox{.8}{
\includegraphics{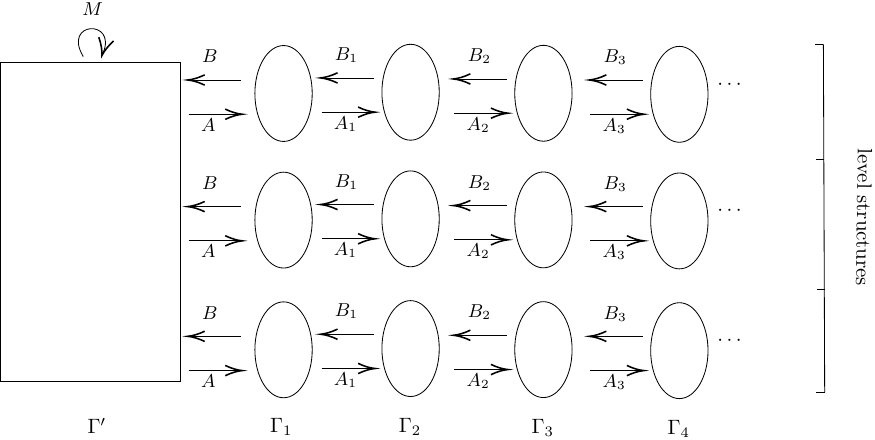}
}
\caption{Graph for $\PGL_2$ unramified at $x$.}
\label{fig:PGL2_unramified}
\end{figure}

Therefore, 
\begin{theorem} \label{t:PGL2_ramified_not_at_x}
    Let $x$ be a point of degree $r$, and let $\Phi_x$ be the corresponding Hecke operator. Consider ramification at divisor $D=\sum_y d_y[y]$, whose support does not contain $x$. Then
    \begin{align*}
r\cdot|\operatorname{Pic}^0(X)(k)|&\cdot
\left|\cH(\cO_D)\setminus G(\cO_D)/T(k)\ltimes U(\cO_D)\right|
\le \dim_\lambda\Gamma \\
&\le \dim_\lambda\Gamma'
+r\cdot|\operatorname{Pic}^0(X)(k)|\cdot
\left|\cH(\cO_D)\setminus G(\cO_D)/T(k)\ltimes U(\cO_D)\right|.
\end{align*}
 These become equalities for all but finitely many $\lambda$.
\end{theorem}

\subsection{Spaces of eigenforms for $\PGL_2$ ramified at $x$: $U$ and $B$ ramification} \label{ss:ramified_PGL2}
We study the graph of the Hecke operator $\Phi_x$ with ramification at $d[x]$ with $H^x\in\{U,B\}$. We have already done the $H^x=1$ case in \cite[Section 4.3]{KZ26}.

As we know, the level structures on a bundle $\cL_1\oplus\cL_2$ with $\deg \cL_1\gg\deg L_2$ are parametrized by 
\begin{itemize}
    \item $H=U$: $U(\cO_{d[x]})\backslash G(\cO_{d[x]})/(T(k)\ltimes U(\cO_{d[x]}))$.
    \item $H=B$: $B(\cO_{d[x]})\backslash G(\cO_{d[x]})/(T(k)\ltimes U(\cO_{d[x]}))$.
\end{itemize}

For $\PGL_2$, dividing by the center $Z$ of $\GL_2$ gives
\begin{itemize}
    \item $H=U$: $Z(\cO_{d[x]})U(\cO_{d[x]})\backslash G(\cO_{d[x]})/(T(k)\ltimes U(\cO_{d[x]}))$.
    \item $H=B$: $B(\cO_{d[x]})\backslash G(\cO_{d[x]})/(T(k)\ltimes U(\cO_{d[x]}))$.
\end{itemize}

 After fixing a trivialization of $\cL_1\oplus\cL_2$ composed of trivializations of $\cL_1$ and $\cL_2$, the level structure is given by a matrix $(a_{ij})$ in $G(\cO_{d[x]})/(T(k)\ltimes U(\cO_{d[x]}))$. Projecting everything to 
 $$
 B(k)\backslash G(k)/B(k)=W=\{\mathrm{id},s\},
 $$ 
 we get a splitting of each double quotient into two components: preimage of $\mathrm{id}$ and of $s$:
\begin{itemize}
    \item {\bfseries $\mathrm{id}$ case:} $a_{21}\in \pi_x\cO_{d[x]}$ for $s\ge 1$. In this case, we can reduce a representative of an element in the quotient to
    \begin{equation}\label{matrix_form_at_identity}
    \begin{pmatrix}
    1&0\\
    \pi_x\cO_{d[x]}&\cO_{d[x]}^\times
\end{pmatrix}.
\end{equation}
    Note that some of these matrices still give same double cosets, but having this form will be enough for our computations.
    \item {\bfseries $s$ case:} $a_{21}\in \cO_{d[x]}^\times$. In this case, there is a slice in the above quotient consisting of matrices
    \begin{equation}\label{eq:matrix_form_at_s}
    \begin{pmatrix}
    0&\cO_{d[x]}^\times/k^\times\\
    1&0
\end{pmatrix}.
\end{equation}
We check that these matrices represent different double cosets. Take two such matrices with entries $a$ and $b$ in the upper-right corner. They represent the same double coset if and only if
$$
Z(\cO_D)U(\cO_D)\mat{0}{a}{1}{0}\cap \mat{0}{b}{1}{0}T(k)U(\cO_D)\ne \varnothing.
$$
We compute general elements of the members of this intersection:
\begin{align*}
    \mat{x}{y}{0}{x}\cdot \mat{0}{a}{1}{0}&=\mat{y}{ax}{x}{0},\qquad x,y\in\cO_{d[x]},\\
\mat{0}{b}{1}{0}\cdot \mat{p}{q}{0}{r}&=\mat{0}{br}{p}{q},\qquad p,r\in k^\times, q\in\cO_{d[x]}.
\end{align*}
Equality of such elements gives $y=q=0$, $x=p\in k^\times$, $a/b=r/x\in k^\times$. This shows that elements (\ref{eq:matrix_form_at_s}) represent different double cosets.
\end{itemize}

To get the result for $B$, we just need to quotient the result for $U$ by $T(\cO_D)$. We get
\begin{itemize}
    \item {\bfseries $\mathrm{id}$ case:} $a_{21}\in \pi_x\cO_{d[x]}$ for $s\ge 1$. In this case, we can reduce a representative of an element in the double quotient to
    \begin{equation}\label{eq:B_matrix_form_at_identity}
    \begin{pmatrix}
    1&0\\
    \pi_x\cO_{d[x]}&1
\end{pmatrix}.
\end{equation}
    Note that some of these matrices still give same double cosets, but having this form will be enough for our computations.
    \item {\bfseries $s$ case:} $a_{21}\in \cO_{d[x]}^\times$. In this case, there is a slice in the above quotient consisting of matrices
    \begin{equation}\label{eq:B_matrix_form_at_s}
    \begin{pmatrix}
    0&1\\
    1&0
\end{pmatrix},
\end{equation}
so we actually get only one coset in this case.
\end{itemize}

Now, we compute the edges in the graph based on which of the two cases we are in. Our calculation will run in the following way. We fix a level structure $a\in\GL_2(\cO_D)$, and the goal is to compute all level structures $b$ in the double coset space connected to $a$ by an edge. The edge corresponding to $\tau\Delta K\in K\Delta K/K$ gives
$$
b=\Delta^{-1}\tau a\sigma,
$$
where $\sigma:\cO^2_{d[x]}\to\cO^2_{d[x]}$ is the inclusion map corresponding to extension of line bundles. Recall that if we choose extension $\cL_1(-x)\oplus\cL_2$, then $\sigma=\smat{\pi_x}{t}{0}{1}$ for $t\in k_x$, and if we choose $\cL_1\oplus\cL_2(-x)$, then $\sigma=\smat{1}{0}{0}{\pi_x}$. In addition, it follows from Theorem \ref{t:hecke_gln_cases} that $\tau$ can be taken in the form $\smat{1}{c}{0}{1}$ for $c\in k_x$.

{\bfseries Case $s$.} We pick $a$ in the form (\ref{eq:matrix_form_at_s}). We start with $\sigma=\smat{1}{0}{0}{\pi_x}$, for which the corresponding line subbundle is $\cL_1\oplus\cL_2(-x)$. Taking $\tau=\smat{1}{c}{0}{1}$, we get
\begin{align*}
    (b_{ij})&=\mat{\pi_x^{-1}}{0}{0}{1}\mat{1}{c}{0}{1}\mat{0}{a_{12}}{1}{0}\mat{1}{0}{0}{\pi_x}\\
    &=\mat{\pi_x^{-1}c}{a_{12}}{1}{0},
\end{align*}
which shows that only $c=0$ works, in which we get the level structure $\mat{0}{a_{12}}{1}{0}$. This implies that in this case, there is exactly one edge of multiplicity $1$ going from each vertex. Moreover, we get all level structures in the $s$-locus by varying $a_{12}$. Thus, using the decomposition from Section \ref{ss:unramified_PGL2}, we see that such edges from the $s$-region to the $s$-region form bijections between the $s$-regions of $\Gamma_i$ and $\Gamma_{i+1}$.

Now, consider the case with $\sigma=\smat{\pi_x}{t}{0}{1}$, for which the corresponding line subbundle is $\cL_1(-x)\oplus\cL_2$. Taking $\tau=\smat{1}{c}{0}{1}$, we get
\begin{align*}
    (b_{ij})&=\mat{\pi_x^{-1}}{0}{0}{1}\mat{1}{c}{0}{1}\mat{0}{a_{12}}{1}{0}\mat{\pi_x}{t}{0}{1}\\
    &=\mat{c}{\pi_x^{-1}(a_{12}+tc)}{\pi_x}{t},
\end{align*}
which implies that $c\equiv -a_{12}/t\pmod{\pi_x}$. This gives exactly one edge for each $t\ne 0$, and all such edges go to the $\mathrm{id}$ locus.

{\bfseries Case $\mathrm{id}$.} We pick $a$ in the form (\ref{eq:B_matrix_form_at_identity}). Take $\sigma=\smat{\pi_x}{t}{0}{1}$. Taking $\tau=\smat{1}{c}{0}{1}$, we get
\begin{align*}
    (b_{ij})&=\mat{\pi_x^{-1}}{0}{0}{1}\mat{1}{c}{0}{1}\mat{1}{0}{a_{21}}{a_{22}}\mat{\pi_x}{t}{0}{1}\\
    &=\mat{1+a_{21}c}{\pi_x^{-1}(t+ta_{21}c+a_{22}c)}{\pi_x a_{21}}{ta_{21}+a_{22}},
\end{align*}
which implies that $t\equiv -\frac{a_{22}c}{1+a_{21}c}\pmod{\pi_x}$. Since $a_{21}\in \pi_x\cO_{d[x]}$, this equation always has a solution. Thus, every $\tau$ gives an edge in this case, so there are no edges going to $\cL_{1}\oplus\cL_{2}(-x)$. Let's do the calculation nevertheless to see the contradiction in this case:
\begin{align*}
    (b_{ij})&=\mat{\pi_x^{-1}}{0}{0}{1}\mat{1}{c}{0}{1}\mat{1}{0}{a_{21}}{a_{22}}\mat{1}{0}{0}{\pi_x}\\
    &=\mat{\pi_x^{-1}(1+a_{21}c)}{a_{22}c}{a_{21}}{\pi_x a_{22}},
\end{align*}
which implies that $1+a_{21}c\in\pi_x\cO_{d[x]}$. This is impossible since $a_{21}\in\pi_x\cO_{d[x]}$ by assumption.

Therefore, the graph looks like on Figure \ref{fig:PGL2_ramified_at_x}.
\begin{figure}
    \centering
\scalebox{.8}{
\includegraphics{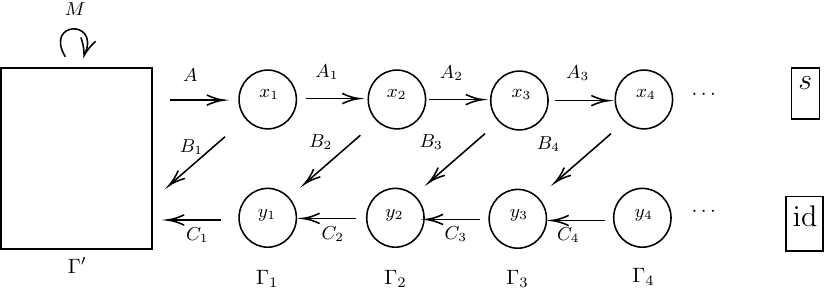}

}
\caption{Graph for $\PGL_2$ ramified at $x$ with $H=B$ or $U$.}
\label{fig:PGL2_ramified_at_x}
\end{figure}

To compute the space of eigenforms, we utilize the methods from \cite[Section 4.2]{KZ26}. Let the matrices corresponding to edges and values of an eigenform be as in Figure \ref{fig:PGL2_ramified_at_x}, with $M$ being the adjacency matrix of $\Gamma'$ and $v$ the value on $\Gamma'$. The eigenvalue condition dictates:
\begin{itemize}
    \item $(M-\lambda)v+Ax_1=0$, 
    \item $\lambda x_1=B_1v+A_1x_2$, $\lambda x_i=B_iy_{i-1}+A_ix_{i+1}$,
    \item $\lambda y_1=C_1v$, $\lambda y_i=C_iy_{i-1}$.
\end{itemize}

We see that any solution of the first equation propagates uniquely to the global solution of all these equations on the entire graph because of $A_i$ being isomorphisms. Therefore, we get the following:
\begin{theorem} \label{t:PGL2_ramified_at_x}
    Let $x$ be a point of degree $r$, $k=\FF_q$, and let the Hecke operator $\Phi_x$ at $x$ be ramified at $d[x]$. Let $\lambda\ne 0$. Then
    \begin{enumerate}[label=(\roman*)]
        \item If $H=U$, then
            $$
        r\cdot|\operatorname{Pic}^0(X)(k)|\cdot\frac{(q^r-1)q^{r(d-1)}}{q-1}\le\dim_\lambda\Gamma\le \dim_\lambda\Gamma'+r\cdot|\operatorname{Pic}^0(X)(k)|\cdot\frac{(q^r-1)q^{r(d-1)}}{q-1}.
            $$
         \item If $H=B$, then
            $$
        r\cdot|\operatorname{Pic}^0(X)(k)|\le\dim_\lambda\Gamma\le \dim_\lambda\Gamma'+r\cdot|\operatorname{Pic}^0(X)(k)|.
            $$
    \end{enumerate}
        
    These are equalities for all but finitely many $\lambda$.
\end{theorem}

Moreover, since $H^x=U$ is regular by Proposition \ref{p:unipotent_implies_regular}, we get that the graph ramified at $D=D'+d_x[x]$ with $H^x=U$ and $d_x\ge 1$ is a covering of the graph on Figure \ref{fig:P1_PGL2_ramified_at_x}. It will retain the same general structure shown on Figure \ref{fig:P1_PGL2_ramified_at_x}. Since $A_i$ are bijections on the graph ramified entirely at $x$, the $s$-locus is a disjoint union of strings homeomorphic to the real line. They are simply connected, hence any covering is trivial. This implies that $A_i$ on the graph ramified at $\cH$ will be bijections as well, hence the same computation applies. We get the following final result.
\begin{theorem}\label{t:H^x=U_dimensions}
    If $\lambda\ne 0$ and $(D,\cH)$ is a ramification datum with $H^x=U$ and $d_x\ge 1$. Then
            $$
        r\cdot|\operatorname{Pic}^0(X)(k)|\cdot\frac{(q^r-1)q^{r(d-1)}}{q-1}\cdot |\cH(\cO_{D'})\backslash G(\cO_{D'})/U(\cO_{D'})|\le\dim_\lambda\Gamma$$
        and
        $$
        \dim_\lambda\Gamma\le \dim_\lambda\Gamma'+r\cdot|\operatorname{Pic}^0(X)(k)|\cdot\frac{(q^r-1)q^{r(d-1)}}{q-1}\cdot |\cH(\cO_{D'})\backslash G(\cO_{D'})/U(\cO_{D'})|.
            $$
\end{theorem}

{\color{black}
Let $\Gamma$ be one of the $\PGL_2$-Hecke graphs of Hecke operators $\Phi$ considered in Theorems~\ref{t:PGL2_ramified_not_at_x}\;,\;\ref{t:PGL2_ramified_at_x}, and\; \ref{t:H^x=U_dimensions}. We generalize these results to the case of generalized eigenspaces.

\begin{theorem}
\label{t:generic_generalized_eigenspaces}
In all the above examples, assume that $\lambda\in\CC\setminus0$ is not an eigenvalue for the adjacency matrix of $\Gamma'$.
Then \begin{equation}\label{eq:generic_surjectivity}
  \Phi-\lambda\cdot\mathrm{id}:\CC^\Gamma\to\CC^\Gamma
\end{equation}
is surjective. In particular, for every $m\geq1$, \begin{equation}\label{eq:generalized_eigenspace_dimension}
  \dim\ker(\Phi-\lambda)^m=m\dim_{\lambda}\Gamma.
\end{equation}
\end{theorem}
\begin{proof}
    We repeat the proof strategy from \cite[Theorem 4.33]{KZ26}. Choose $g\in\CC^\Gamma$ and write it as $(g',g_i^x,g_i^y)$ with respect to the decomposition of $\Gamma$ into $\Gamma'$, $x_i$, and $y_i$ loci. Adapting the arguments before Theorem \ref{t:PGL2_ramified_at_x}, we get that the equation
    $$
    (\Phi-\lambda)f=g
    $$
    is equivalent to the system
    \begin{itemize}
    \item $(M-\lambda)v+Ax_1+g'=0$, 
    \item $\lambda x_1+g_1^x=B_1v+A_1x_2$, $\lambda x_i+g_i^x=B_iy_{i-1}+A_ix_{i+1}$,
    \item $\lambda y_1+g_1^y=C_1v$, $\lambda y_i+g_i^y=C_iy_{i-1}$.
\end{itemize}
As before, we see that any solution of the first equation extends uniquely to a solution of the whole system. In turn, the first equation has a solution since $M-\lambda$ is assumed to be invertible.

The second part follows from the exact sequence
$$
    0\to \ker(\Phi-\lambda)\to \ker(\Phi-\lambda)^k\xrightarrow{\Phi-\lambda}\ker(\Phi-\lambda)^{k-1}\to 0
    $$
and induction, where the rightmost surjectivity was established above.
\end{proof}
}

\begin{remark}
    With the family of Eisenstein series defined in \ref{s:Eisenstein}, its differentials form a basis of the generalized eigenspace of $\Phi_x$. We refer to \cite[Section 4.5]{KZ26} for details.
\end{remark}

\begin{example}
{\rm
    Let $X=\PP^1$, $\deg x=1$, $d=1$. We notice that $T(k)\ltimes U(\cO_{d[x]})$ equals $B(k)$, and both double quotients are parametrized by two matrices: $\smat1001$ and $\smat0110$. It only remains to compute the outgoing edges of $\cO\oplus\cO$. We need to use the case $\sigma=\smat100{\pi_x}$ to produce $\cO\oplus\cO(-x)$. The above computations show that such edges must go to the $s$-locus, and hence it will be one edge of multiplicity $q$. The final graph is shown on Figure \ref{fig:P1_PGL2_ramified_at_x}.

    \begin{figure}
        \centering
        \scalebox{.9}{
        
\includegraphics{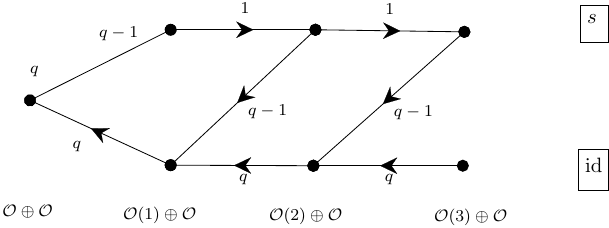}
        
        }
\caption{Ramified graph for $G=\PGL_2$ and $X=\PP^1$ at $1\cdot[x]$, $\deg x=1$, $H=U$ or $B$.}
\label{fig:P1_PGL2_ramified_at_x}
    \end{figure}
}
\end{example}

\subsection{$\PGL_2$ examples ramified at several points}
In this section, we describe explicit examples with tame $B$-ramification for several points on $\PP^1$ and $\PGL_2$. We first develop the general method for doing so.

We fix pairwise distinct points  $x,z_1,\ldots, z_g\in \PP^1(k)$ of degree $1$, and we declare the ramification divisor to be $D:=\sum_{i}[z_i]$. Therefore, $\cO_D=k^g$. The Hecke operator is considered at $x$. For simplicity of the exposition, we assume that none of these points coincides with $\infty$, which allows us to treat them as elements of $k$.

On each bundle $\cE_n:=\cO(n)\oplus \cO$, choose a trivialization of $\cO(n)$ over $D$. It induces a trivialization of $\cE_n$, under which the level structures are parametrized by
$$
B(\cO_D)\backslash\GL_2(\cO_D)/\Aut(\cE_n)_D\simeq \PP^1(k)^g/\Aut(\cE_n)_D,
$$
where $\Aut(\cE_n)_D$ denotes the image of $\Aut(\cE_n)_D$ under the restriction map
$$
\Aut(\cE_n)\to \Aut(\cE_n|_D)\simeq \GL_2(\cO_D).
$$

Since $\Aut(\cE_0)=\GL_2(k)$ and
$$
\Aut(\cE_n)=\mat{k^\times}{H^0(X,\cO(n))}{0}{k^\times}
$$
for $n\ge 1$, realizing the sections of $H^0(X,\cO(n))$ as polynomials in one variable of degree at most $n$, we can write
$$
\Aut(\cE_n)_D=T(k)\ltimes\left\{\left(\smat{1}{a_nz_i^n+\ldots+a_1z_i+a_0}{0}{1}\right)_{i=1..g}:a_j\in k\right\}\subset T(k)\ltimes U(k)^g.
$$
In particular, by a Vandermonde determinant argument for $n\ge g-1$ we get
$\Aut(\cE_n)_D=T(k)\ltimes U(k)^g$, and the Bruhat decomposition gives
$$
B(\cO_D)\backslash\GL_2(\cO_D)/\Aut(\cE_n)_D=B(\cO_D)\backslash\GL_2(\cO_D)/T(k)\ltimes U(\cO_D)\simeq \{0,\infty\}^g.
$$

In general, since the level structures are represented by points in $\PP^1(k)^g$, it is convenient to describe Hecke modifications in this language. Recall that the isomorphism of $B\backslash\GL_2$ with $\PP^1$ is done via
$$
Ba\mapsto (0,1)\cdot a,
$$
and we treat $[0:1]$ as the infinity point. If we use an affine coordinate $v$ to represent a point $[1:v]\in\PP^1(k)$, then the action of $\Aut(\cE_n)_D$ identifies
\begin{equation}\label{eq:automorphisms_action_on_P^1}
    (v_1,\ldots,v_n)\sim (\lambda v_1+f(z_1),\ldots, \lambda v_n+f(z_n)) 
\end{equation}
for $\lambda\in k^\times$ and
$$
f(z)\in k\oplus kz\oplus\ldots\oplus kz^n\subset k[z].
$$
This formula also applies when some of $v_i=\infty$ if we declare that $\lambda \infty+\mu=\infty$ for all $\lambda\in k^\times$ and $\mu\in k$.

Recall that Hecke modifications are represented by inclusion matrices
$$
\sigma=\mat{\pi_x}{c}01\text{ or }\mat{1}{0}{0}{\pi_x},\qquad c\in k.
$$
At $z\in\PP^1\setminus x$, these matrices take the form
$$
\sigma=\mat{z-x}{c}01\text{ or }\mat{1}{0}{0}{z-x},\qquad c\in k,
$$
transporting a level structure $[v_1:v_2]$ at $z$ to $[v_1:v_2]\sigma$, i.e. to
$$
[(z-x)v_1:cv_1+v_2]\text{ or }[v_1:(z-x)v_2],\qquad c\in k.
$$
If we just use the affine coordinate $[1:v]$ for the level structure ($v=v_2/v_1$), it transports to
$$
\frac{v+c}{z-x}\quad\text{ or }\quad(z-x)v,\qquad c\in k.
$$
Note that $\infty$ is connected to $\infty$ and $\AA^1$ is connected to $\AA^1$.

In our setting, the level structures $(v_1,\ldots, v_g)$ at $z_1,\ldots,z_g$ transport to
$$
\left(\frac{v_i+c}{z_i-x}\right)_{i=1..g}\quad\text{ or }\quad\big((z_i-x)v_i\big)_{i=1..g},\qquad c\in k.
$$
For $l=0\ldots g$, define $V_l$ to be the set of points in $\PP^1(k)^g$ with exactly $k$ coordinates equal to $\infty$. These subsets are clearly preserved by $T(k)\ltimes U(k)^g$, hence their projections onto $\PP^1(k)^g/\Aut(\cE_n)_D$ do not intersect. By the above description of arrows, elements of $V_l$ are connected only to elements of $V_l$, unless the target is $\cE_0$. Schematically, the graph looks like on Figure \ref{fig:B-ramif_at_many_points}.

\begin{figure}
        \centering
        \scalebox{.8}{
            \includegraphics{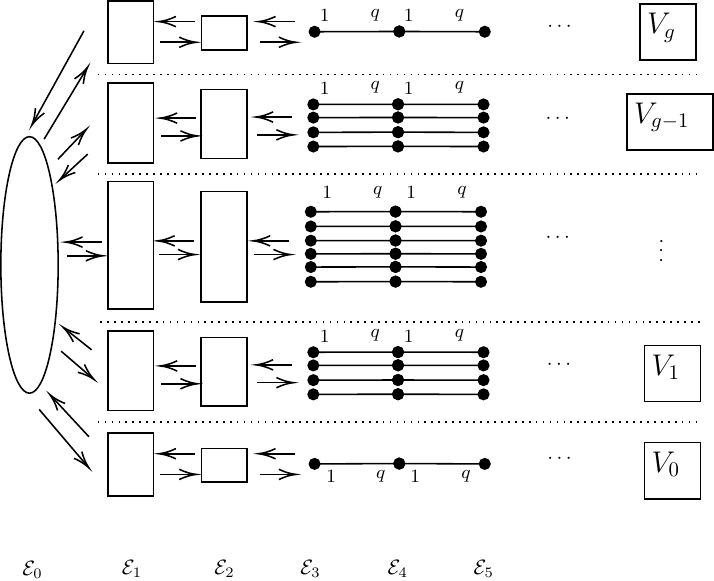}
        }
\caption{$B$-ramified graph for $G=\PGL_2$ and $X=\PP^1$ at $D=\sum_i[z_i]$, $\deg x=\deg z_i=1$. Example for $g=4$.}
\label{fig:B-ramif_at_many_points}
    \end{figure}

Now, we describe Hecke modifications of level structures on $\cE_0$. Hecke modifications of the bundle are given by inclusions $\cO\oplus\cO(-1)\to\cO\oplus\cO$ represented by inclusion matrices
$$
\sigma=\mat{1}{0}c{\pi_x}\text{ or }\mat{0}{\pi_x}{1}{0},\qquad c\in k
$$
transporting level structures $(v_1,\ldots, v_g)$ at $z_1,\ldots,z_g$ to
$$
\left(\frac{z_i-x}{v_i^{-1}+c}\right)_{i=1..g}\quad\text{ or }\quad\left(\frac{z_i-x}{v_i}\right)_{i=1..g},\qquad c\in k,
$$
where we use the convention $\infty^{-1}=0$.

With these preparations, we are ready to compute small examples. We start with $g=1,2$. In this case, the cusp locus starts with $\cE_1$, so we only need to compute edges involving $\cE_0$.

\begin{example}[Case $g=1$]
    {\rm
    In this case, all level structures on $\cE_1$ are equivalent, so the edges going from $\cE_1$ to $\cE_0$ are all of degree $q$. Without loss of generality, we can assume that the level structure on $\cE_0$ is given by $0\in\PP^1(k)$. By the above computation, it is connected to the level structure $0$ with multiplicity $q$ and to the level structure $\infty$ with multiplicity $1$. The graph is presented on Figure \ref{fig:B-ramif_at_1_point}

\begin{figure}
        \centering
        \scalebox{.9}{
        \includegraphics{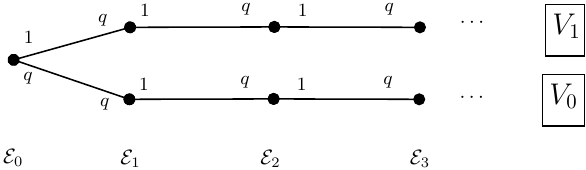}
        }
\caption{$B$-ramified graph for $G=\PGL_2$ and $X=\PP^1$ at $g=1$ point $z_1$, $\deg x=\deg z_1=1$.}
\label{fig:B-ramif_at_1_point}
    \end{figure}

    }
\end{example}

\begin{example}[Case $g=2$]
{\rm
In this case, we have two equivalence classes of level structures on $\cE_0$: $(0,0)$ and $(0,\infty)$. The orbit of the first (second) structure is precisely collections of two coinciding (distinct) points on $\cP^1(k)$. We start with computing the edges from $\cE_1$ to $\cE_0$. For simplicity, we may assume that $x=0$ in the above formulas. A level structure $(v_1,v_2)\in \{0,\infty\}^2$ on $\cE_1$ is then connected to
$$
\left(\frac{v_1+c}{z_1},\frac{v_2+c}{z_2}\right).
$$
It equals the $(0,0)$ level structure if and only if
$$
\frac{v_1+c}{z_1}=\frac{v_2+c}{z_2}.
$$
This is true for all $c\in k$ if $(v_1,v_2)=(\infty,\infty)$, has exactly one solution in $c$ for $(v_1,v_2)=(0,0)$ and has no solutions otherwise. The rest of the edges (recall that the number of edges from a level structure on $\cE_1$ to level structures on $\cE_0$ equals $q$) go into $(0,\infty)$.

Now, we compute edges from $\cE_0$ to $\cE_1$. Using the above formula, the edges from $(0,0)$ go to  $(0,0)$ with multiplicity $q$ and to $(\infty,\infty)$ with multiplicity $1$. The edges from $(0,\infty)$ go to $(0,z_1/c)$ with multiplicity $1$ for each $c\in k$ and to $(\infty,0)$ with multiplicity $1$. However, using (\ref{eq:automorphisms_action_on_P^1}), we see that the points $(0,z_1/c)$ are identified with $(0,0)$ for $c\ne 0$ and with $(0,\infty)$ for $c=0$. The graph is shown on Figure \ref{fig:B-ramif_at_2_points}.

\begin{figure}
        \centering
        \scalebox{.9}{
            \includegraphics{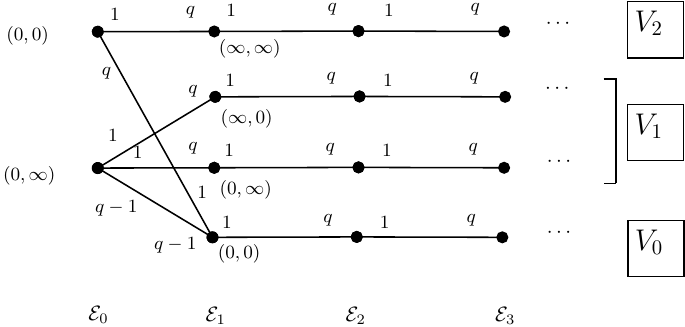}

        }
\caption{$B$-ramified graph for $G=\PGL_2$ and $X=\PP^1$ at $g=2$ points $z_i$, $\deg x=\deg z_i=1$.}
\label{fig:B-ramif_at_2_points}
    \end{figure}

\begin{figure}
        \centering
        \scalebox{.9}{
        
    \includegraphics{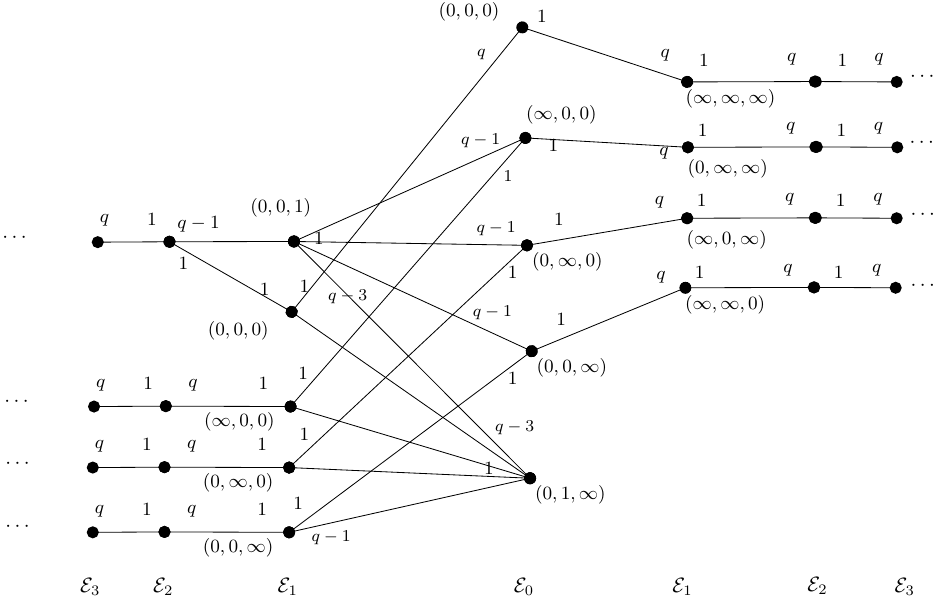}

        }
\caption{$B$-ramified graph for $G=\PGL_2$ and $X=\PP^1$ at $g=3$ points $z_i$, $\deg x=\deg z_i=1$. Some of edge multiplicities are not shown for readability and can be recovered from the fact that the outgoing degree of each vertex is $q+1$.}
\label{fig:B-ramif_at_3_points}
    \end{figure}
}
\end{example}

\begin{example}[Case $g=3$]
{\rm
Here, the cusp starts from $\cE_2$, so we expect more than $8$ level structures on $\cE_1$. Let $(v_1,v_2,v_3)\in\PP^1(k)$ represent a level structure. First, assume that it lies in $V$, i.e. at least one of $v_i$ equals $\infty$. Without loss of generality, assume that $v_1=\infty$. Then (\ref{eq:automorphisms_action_on_P^1}) tells that
$$
(\infty,v_2,v_3)\sim (\infty,\lambda v_2+f(z_2),\lambda v_3+f(z_3))
$$
for $\lambda\in k^\times$ and $f$ a linear polynomial. Using Vandermonde determinant argument (or simple manual calculation), we see that any such level structure can be brought to the form $\{0,\infty\}^3$. This gives $7$ orbits in $V$, as before.

Now, assume that all $v_i$ are finite, i.e. the level structure lies in $U$. We claim that there are precisely two level structures in this case: $(0,0,0)$ and $(0,0,1)$. Indeed, for $a,b\in k$ and $\lambda\in k^\times$,
\begin{align*}
    (0,0,1)&\sim (az_1+b,az_2+b,\lambda+az_3+b),\\
    (0,0,0)&\sim (az_1+b,az_2+b,az_3+b).
\end{align*}
We can easily see that if we allow $\lambda$ to be $0$ as well, then we can get any element of $U$ from $(0,0,1)$. Moreover, the orbit of $(0,0,0)$ is precisely the case $\lambda=0$, so these two orbits cover the entire $U$. It remains to prove that they do not intersect. For this, fix a non-zero triple of elements $(r_1,r_2,r_3) \in k^3$ such that 
$$
\sum_i r_i=\sum_ir_iz_i=0.
$$
Such a triple is unique up to simultaneous rescaling by $k^\times$. Then we can easily see that
\begin{align*}
    (0,0,1)\text{-orbit}&=\left\{(v_1,v_2,v_3):\sum_i r_i v_i\ne0\right\},\\
    (0,0,0)\text{-orbit}&=\left\{(v_1,v_2,v_3):\sum_i r_i v_i=0\right\}.
\end{align*}
This finishes the proof and gives a nice criterion to check which of the classes our level structure lies in. In summary, we have $9$ level structures: $7$ usual ones in $V$ and $(0,0,0)$, $(0,0,1)$ in $U$.

We may assume that $x=0$. To compute maps from $\cE_2$ to $\cE_1$, choose a level structure $(v_1,v_2,v_3)\in \{0,\infty\}^3$ on $\cE_2$. The targets of edges from it are
$$
\left(\frac{v_i+c}{z_i}\right)_{i=1..3},\qquad c\in k.
$$
If at least one of $v_i=\infty$, the target lands in $V$, and therefore there will be a single edge of multiplicity $q$.

If $(v_1,v_2,v_3)=(0,0,0)$, the targets are
$$
\left(\frac{c}{z_1},\frac{c}{z_2},\frac{c}{z_3}\right),\qquad c\in k.
$$
Since the matrix 
$$
\begin{pmatrix}
    1&1&1\\
    z_1&z_2&z_3\\
    1/z_1&1/z_2&1/z_3
\end{pmatrix}
$$
is invertible, it has no non-trivial kernel $(r_1,r_2,r_3)$ as above. Therefore, there will be precisely $q-1$ edges to $(0,0,1)$ (case $c\ne 0$) and $1$ edge to $(0,0,0)$ (case $c=0$).

Now, we compute edges from $\cE_1$. We start with edges to $\cE_2$. Each level structure $(v_i)$ has precisely one edge to $\cE_2$ with target
$$
\left(z_1v_1,z_2v_2,z_3v_3\right).
$$
We can easily see that the graph looks precisely like in the cusp locus for vertices in $V$, and both vertices $(0,0,0)$ and $(0,0,1)$ have one edge of multiplicity $1$ to $(0,0,0)$.

Now, we compute edges from $\cE_1$ to $\cE_0$. For this, we need to classify level structures on $\cE_0$. Since any three distinct points on $\PP^1(k)$ can be moved to any other three distinct points on $\PP^1(k)$ by the $\PGL_2(k)$-action, the level structures on $\cE_0$ are completely classified by which points coincide. We choose representatives in the following way:
$$
(0,0,0),(0,0,\infty),(0,\infty,0),(\infty,0,0), (0,1,\infty).
$$
So, we need to study which of the coordinates in 
$$
\left(\frac{v_1+c}{z_1},\frac{v_2+c}{z_2},\frac{v_3+c}{z_3}\right),\qquad c\in k.
$$
coincide for a level structure $(v_i)$ on $\cE_1$. We have the following cases:
\begin{itemize}
    \item $(0,0,0)$: this gives pairwise distinct coordinates for $c\ne 0$ and equal coordinates for $c=0$.
    \item $(0,0,1)$: for three values of $c$, this gives the three cases when two out of three coordinates are equal. For the rest of $c$, the coordinates are pairwise distinct.
    \item $(\infty,0,0)$ and permutations: this gives pairwise distinct coordinates for $c\ne 0$ and the last two coordinates equal for $c=0$.
    \item $(\infty,\infty,0)$ and permutations: this gives only the first two coordinates equal for any $c$.
    \item $(\infty,\infty,\infty)$: all coordinates are equal for any $c$.
\end{itemize}

It remains to compute edges from the $5$ level structures on $\cE_0$. Recall that to $(v_1,v_2,v_3)$ we connect
$$
\left(\frac{z_1}{v_1^{-1}+c},\frac{z_2}{v_2^{-1}+c},\frac{z_3}{v_3^{-1}+c}\right)\quad\text{ or }\quad\left(\frac{z_1}{v_1},\frac{z_2}{v_2},\frac{z_3}{v_3}\right),\qquad c\in k,
$$
\begin{itemize}
    \item $(0,0,0)$: we get $(0,0,0)$ for any $c\in k$ or $(\infty,\infty,\infty)$.
    \item $(0,0,\infty)$ and permutations: we get $(0,0,z_3/c)$ for any $c\in k$ or $(\infty,\infty,0)$. Note that $(0,0,z_3/c)\sim (0,0,1)$ if $c\ne 0$ and $(0,0,z_3/c)=(0,0,\infty)$ if $c=0$.
    \item $(0,1,\infty)$: we get $(0,z_2/(1+c),z_3/c)$ for any $c\in k$ or $(\infty,z_2,0)\sim (\infty,0,0)$. We can easily see that the condition $(0,z_2/(1+c),z_3/c)\sim (0,0,0)$ is a linear equation on $c$ having a unique solution $c\ne 0,-1$. In summary, this case gives one edge to each of $(0,0,0)$, $(0,\infty,0)$, $(0,0,\infty)$, and $(\infty,0,0)$, and $q-3$ edges to $(0,0,1)$.
\end{itemize}

Summing everything up, we get a graph as on Figure \ref{fig:B-ramif_at_3_points}.
}
\end{example}

\begin{example}[$T$ at $x$]
    Here, we give an example for $X=\PP^1$, $G=\PGL_2$, $H^x=T$. By Remarks \ref{r:strengthening_of_general_connections} and \ref{r:change_of_ramification}, this graph is obtained by identifying vertices and keeping edges in the example of $\operatorname{id}$-ramification at $x$. This example was computed in \cite[Example 4.31]{KZ26} and is presented on Figure \ref{fig:PGL2_P1_ramified_example}.

     Thus, we need to understand which vertices need to be identified. The level structures of the $T$-ramified graph are given by
    $$
    T(k)\backslash G(k)/B(k)=T(k)\backslash \PP^1(k)=\{0,1,\infty\},
    $$
    which means that the points in the $\PP^1(k)\setminus\{0,\infty\}$ on Figure \ref{fig:PGL2_P1_ramified_example} need to be identified. The desired graph is presented on Figure \ref{fig:PGL2_P1_T_ramified_example}.
    \begin{figure}
\centering
\scalebox{.9}{
\includegraphics{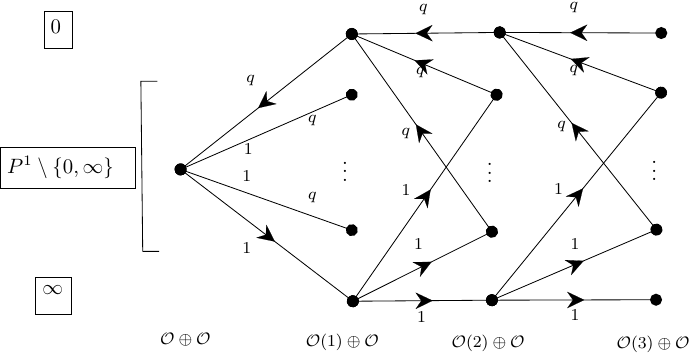}
}
\caption{Ramified graph for $G=\PGL_2$ and $X=\PP^1$ at $1\cdot[x]$, $\deg x=1$.}
\label{fig:PGL2_P1_ramified_example}
\end{figure}

\begin{figure}
\centering
\scalebox{.9}{
\includegraphics{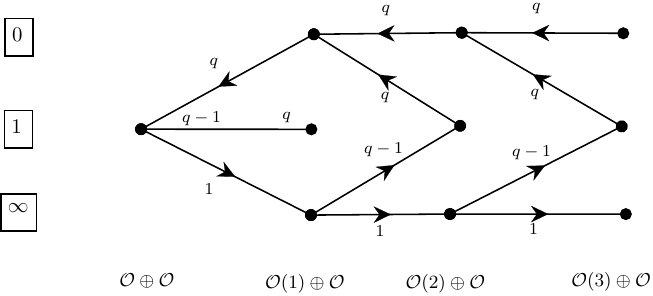}
}
\caption{Ramified graph for $G=\PGL_2$ and $X=\PP^1$ at $1\cdot[x]$, $\deg x=1$ with $H^x=T$.}
\label{fig:PGL2_P1_T_ramified_example}
\end{figure}

\end{example}

\begin{example}[$T$ at $x=0$ and $U$ at $y=1$]
    {\rm
    \begin{figure}
\centering
\scalebox{.8}{

\includegraphics{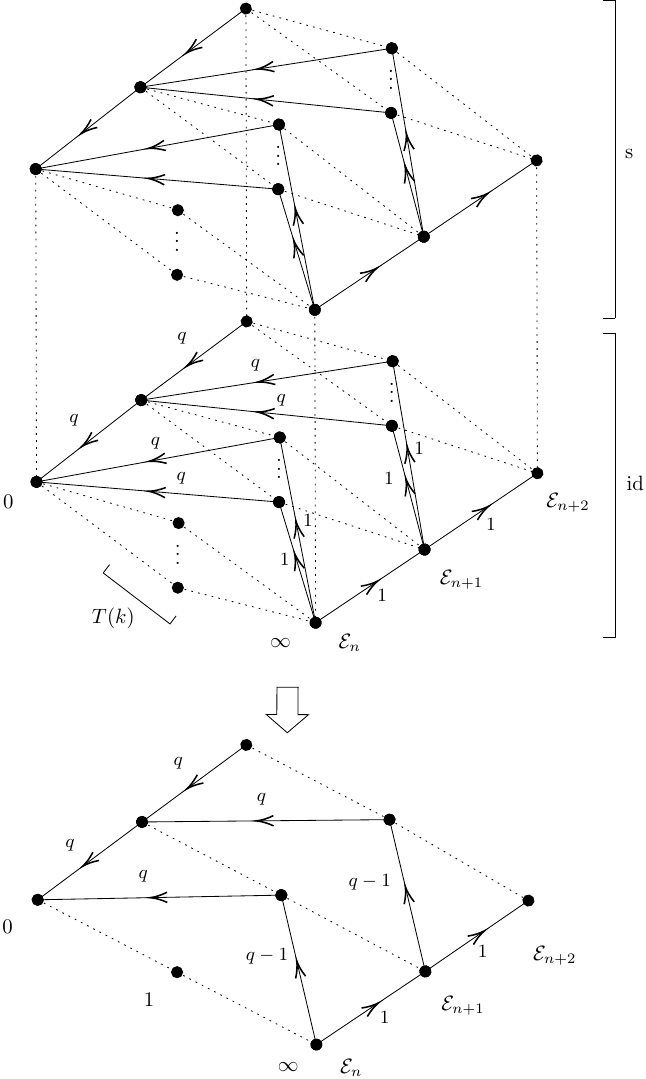}

}
\caption{Cusp of the ramified graph for $G=\PGL_2$ and $X=\PP^1$ at $[x]+[y]$ with $H^x=T$ and $H^y=U$. The projection to the $T$-ramified graph at $[x]$ is shown. Punctured lines are drawn to improve visualization by grouping level structures over single bundles. Edge multiplicities of the $s$-part are the same as of the $\mathrm{id}$-part.}
\label{fig:PGL2_P1_T_ramified_at_x_and_y}
\end{figure}
    Here, we give an example for $X=\PP^1$, $G=\PGL_2$, $x=0$, $y=1$, $H^x=T$ and $H^y=U$. This will be a counterexample to \ref{t:covering_map_GL} in the case when the regularity condition is not satisfied.

    By Remarks \ref{r:strengthening_of_general_connections} and \ref{r:change_of_ramification}, this graph is obtained by identifying vertices and keeping edges in the example of $\operatorname{id}$-ramification at $x$ and $y$. This example was computed at the cusp in \cite[Example 4.32]{KZ26}, and all connections are summarized in
    \begin{align*}
    (a,\infty)&\longrightarrow (a,\infty),\\
    (a,\infty)&\longrightarrow (c,at_{-c^2}),\\
    (a,p)&\longrightarrow (a,0),\quad p\in \PP^1(k)\setminus\infty,
\end{align*}
    where $a\in G/U$ represents the level structure over $y$ and $t_c:=\smat{c}{0}{0}{1}$.

    To understand which vertices need to be identified, we compute the fibers over $\{0,1,\infty\}\in T\backslash G/B$ in 
    $$
    (U\backslash G/U\times T\backslash G/U)/T.
    $$
    Note that we can identify $T\backslash G/U$ with $\PP^1(k)$ and the right $T$-action with the scaling action. Then the points $0$ and $\infty$ are stabilized and $T$ acts on $\PP^1\setminus \{0,\infty\}$ freely. Therefore, the fibers are
    \begin{align*}
        p^{-1}(0)&\simeq U\backslash G/B\simeq \{\operatorname{id},s\},\\
        p^{-1}(\infty)&\simeq U\backslash G/B\simeq \{\operatorname{id},s\},\\
        p^{-1}(1)&\simeq U\backslash G/U\simeq \{\operatorname{id},s\}\times T(k).\\
    \end{align*}
    This gives the following connections:
    \begin{align*}
    (a,\infty)&\longrightarrow (a,\infty),\\
    (a,\infty)&\longrightarrow (c,at_{-c^2})\sim(1,at_{-c}),\\
    (a,p)&\longrightarrow (a,0),\quad p\in \{0,1\}.
\end{align*}

The final graph is shown in Figure \ref{fig:PGL2_P1_T_ramified_at_x_and_y}. We can easily see that in general, there is no control over ingoing edges: For example, we see that above $q$ edges between vertices of levels $1$ and $0$ in the graph downstairs lie $q(q-1)$ edges to a given lift of the level $0$ vertex, while above $q-1$ edges  between vertices of levels $\infty$ and $1$ in the graph downstairs lies only one edge to a given lift of the level $1$ vertex. In particular, the projection map is not a topological covering.
   
    }
\end{example}

\section{Connections to the theory of Eisenstein series for $\PGL_2$}\label{s:Eisenstein}
We connect our findings to the known theory of Eisenstein series for $\PGL_2$. We will follow \cite{Li79}, where the dimensions of spaces of Eisenstein series were computed.

First of all, we need to recall the necessary preliminaries. Fix a finite index subgroup $K'$ of $K=G(\cO)$. Let $V:=\CC[K/K']$ be the set of functions on $K/K'$. This vector space has a linear right action of $K$ induced from the standard action on $K/K'$. Recall the theorem of Harder:
\begin{theorem}\label{t:Harder_iso}
    There is an isomorphism of vector spaces
    \begin{equation}\label{eq:Harder_iso}
        \operatorname{Fun}(G(F)\backslash G(\AA),\CC[K/K'])^K\simeq \operatorname{Fun}(G(F)\backslash G(\AA),\CC)^{K'}=\operatorname{Fun}(G(F)\backslash G(\AA)/K',\CC).
    \end{equation}
\end{theorem}
\begin{proof}
    The isomorphisms are constructed in the following way. Let $\mathbbm 1_{gK'}\in\CC[K/K']$ be the characteristic function of the coset $gK'$. These elements form a basis of $\CC[K/K']$, which gives a projection operator
    $$
    \pi:\CC[K/K']\to \CC\mathbbm 1_{K'}=\CC.
    $$
    
    To $f\in \operatorname{Fun}(G(F)\backslash G(\AA),\CC[K/K'])^K$, we associate the function $\pi\circ f\in \operatorname{Fun}(G(F)\backslash G(\AA),\CC)^{K'}$. In another direction, to a function  $f\in \operatorname{Fun}(G(F)\backslash G(\AA),\CC)^{K'}$, we associate the function $f'\in \operatorname{Fun}(G(F)\backslash G(\AA),\CC[K/K'])^K$ defined by
    $$
    f'(x):=\sum_{gK'\in K/K'}f(xg)\mathbbm 1_{gK'}.
    $$
    It is easy to see that these two maps are mutually inverse.
\end{proof}
Therefore, we can think of ramified automorphic forms as unramified vector-valued automorphic forms. This is the language that we are going to utilize.

 Pick a point $x\in |X|$ such that $K'_x=K_x$. In other words, we assume no ramification at $x$. The Hecke operator $\Phi_x$ acts on both sides of \ref{eq:Harder_iso} by
$$
(\Phi_x f)(g)=\int_{K_x\Delta_x K_x}f(gh_x)\,d\mu(h_x).
$$
We can easily see that the isomorphism \ref{eq:Harder_iso} commutes with these actions.


As in \cite{Li79}, define the spaces of invariants
\begin{align*}
    V^U&:=V^{U(\cO)}\simeq \CC[U(\cO)\backslash K/K'],\\
    V^U_0&:=V^{T(\FF_q)U(\cO)}\simeq \CC[T(\FF_q)U(\cO)\backslash K/K'].
\end{align*}
The action of $T(\cO)$ ($T(\cO)/T(\FF_q)$) on $V^U$ ($V^U_0$) is semisimple, which gives a decomposition
$$
V^U=\bigoplus_{\chi:T(\cO)\to \CC^\times}V_\chi^U\qquad \left(V^U_0=\bigoplus_{\chi:T(\cO)/T(\FF_q)\to \CC^\times}V_\chi^U\right),
$$
where 
$$
V_\chi^U:=\{v\in V^U:\forall_{t\in T(\cO)}\, t.v=\chi(t)v\}.
$$

Fix a character $\chi:T(\FF_q)\backslash T(\cO)\to \CC^\times$. Define the space
$$
W_\chi:=\{\eta\in\operatorname{Fun}(T(F)\backslash T(\AA),\CC):\forall_{t\in T(\cO),g\in T(\AA)}\, \eta(gt)=\chi(t)\eta(g)\}\simeq \operatorname{Ind}_{T(\FF_q)\backslash T(\cO)}^{T(F)\backslash T(\AA)}\CC_\chi,
$$
where $\CC_\chi$ is the one-dimensional representation of $T(\cO)$ associated with $\chi$. Because of the presentation as the induced representation, we can find a vector space isomorphism
\begin{equation}\label{eq:W_chi_dimension}
W_\chi\simeq \operatorname{Fun}(T(F)\backslash T(\AA)/T(\cO),\CC).
\end{equation}

For such $\chi$ and $\eta\in W_\chi$, choose $f\in V_\chi^U$ and define the associated {\bfseries Eisenstein series} $E(-;\eta,f)$ in the following way:
\begin{itemize}
    \item Identifying $T(F)\backslash T(\AA)$ with $B(F)U(\AA)\backslash B(\AA)$ and $\CC$ with $\CC f\in V_{\chi}^U$, we view $\eta$ as a function from the latter set to $V_{\chi}^U$.
    \item Using the Iwasawa decomposition $G(\AA)=B(\AA)K$, we extend $f$ to a function
    $$
    \tilde\eta:B(F)\backslash G(\AA)/K\to V^U
    $$
    by
    $$
    \tilde\eta(bgkk'):=k.\eta(g),\qquad b\in B(F),\,g\in B(\AA), k\in K,\, k'\in K'.
    $$
    \item We set
    $$
    E(g;\eta,f):=\sum_{a\in B(F)\backslash G(F)}\tilde\eta(ag).
    $$
\end{itemize}
\begin{proposition}[\cite{Li79}, Proposition 2.1]
    The above construction of $\tilde\eta$ is well-defined.
\end{proposition}

Now, we need to argue about the well-definiteness of the Eisenstein series. For this, we study the space $W_\chi$ a bit closer. First of all, there is a degree function
\begin{align*}
 \deg:T(\AA)&\to \ZZ,\\
 t&\mapsto \sum_{y\in |Y|}\deg(y)\operatorname{val}_{\pi_y}(t_y)
\end{align*}
We denote the kernel of this map by $T(\AA)_0$. Therefore, we have an exact sequence
$$
0\to T(\AA)_0\to T(\AA)\xrightarrow{\deg}k\ZZ\to 0
$$
for some $k\in\ZZ_{>0}$ ($k=1$ when the curve $X$ has a point over $\FF_q$). We can split this sequence by choosing an element in $T(\AA)$ of degree $k$. We will fix such a splitting. This also gives rise to a splitting
$$
T(F)\backslash T(\AA)\simeq T(F)\backslash T(\AA)_0\oplus\ZZ.
$$

In particular, since $T(\cO)\subset T(\AA)_0$, any $\eta$ in $W_\chi$ has a decomposition $(\eta_0,\tau)$, where $\eta_0$ is a function on $T(F)\backslash T(\AA)_0$ and $\tau\in\CC$. We will write $\tau=q^{-s}=e^{-s\log q}$ for $s\in\CC$, $\eta_{0,s}$ for $(\eta_0,s)$, and $E_s(-;\eta_0,f)$ for $E(-;\eta_{0,s},f)$. We have the following:

\begin{theorem}[\cite{Li79}, Theorem 2.3]
    Assume that $\mathrm{Re}(s)>1$. Then the sum defining $E_s(-;\eta_0,f)$ converges to a holomorphic function in $s$. Moreover,
    $$
    \Phi_xE_s(-;\eta_0,f)=(\eta_{0,s}(\pi_x)+q^{\deg x}\eta_{0,s}(\pi_x^{-1}))  E_s(-;\eta_0,f).
    $$
\end{theorem}

Now, pick a generic $\lambda\in \CC$. We would like to understand how many $\lambda$-eigenforms are provided by Eisenstein series. Pick $\eta_0$ as above. The equation
$$
\eta_{0,s}(\pi_x)+q^{\deg x}\eta_{0,s}(\pi_x^{-1})=\lambda
$$
is quadratic in $q^{-s\deg x}$ and symmetric under $s\leftrightarrow 1-s$, which gives precisely $\deg x$ solutions with $\Re s>1$. Similarly to (\ref{eq:W_chi_dimension}), we get that the space of all $\eta_0$ can be identified with
$$
\operatorname{Fun}(T(F)\backslash T(\AA)_0/T(\cO),\CC)\simeq \operatorname{Fun}(\Pic^0(X)(\FF_q),\CC).
$$
In addition, we have $\dim V_{\chi_s}^U$ choices of $f$, where $\chi_s$ is the restriction of $\eta_{0,s}$ to $T(\cO)$. Similarly to (\ref{eq:W_chi_dimension}), this vector space is isomorphic to $V_{\chi=1}^U$, the space associated with the trivial character $1$. Therefore, we have just counted
$$
\deg x\cdot |\Pic^0(X)(\FF_q)|\cdot\dim V^U_0
$$
Eisenstein series. The following theorem completes the picture.
\begin{theorem}[\cite{Li79}, Theorem 7.1]\label{t:Li_dimension_formula}
    For all but finitely many $s$, the space of all $\lambda_s$-eigenforms of $\Phi_x$ decomposes as the direct sum of the spaces of Eisenstein series and cusp forms. Moreover, the subspace of Eisenstein series has dimension
    $$
    \deg x\cdot |\Pic^0(X)(\FF_q)|\cdot\dim V^U_0.
    $$
\end{theorem}

Note that
$$
\dim V^U_0=|T(\FF_q)U(\cO)\backslash K/K_{\cH}(D)|=|T(\FF_q)U(\cO_D)\backslash G(\cO_D)/\cH(\cO_D)|.
$$
Thus, we observe the match between the dimension formulas given in Theorems \ref{t:PGL2_ramified_not_at_x} and \ref{t:Li_dimension_formula}. On the other hand, approach in \cite{Li79} does not tackle the case considered when the Hecke operator is taken in a ramified point. Theorems \ref{t:H^x=U_dimensions} and \ref{t:PGL2_ramified_at_x} provide dimension formulas in this case, which are clearly different from the dimensions given by Eisenstein series. Therefore, the Eisenstein series construction does not immediately give eigenforms for $\Phi_x$ when $x$ is ramified. In future work, we will give a way to construct eigenforms for $\Phi_x$ from Eisenstein series in this case. Such construction will turn out to produce eigenforms for the global Hecke algebra.

\bibliographystyle{alphaurl}
\bibliography{references}

@article{AB09,
  author  = {Arkhipov, Sergey and Bezrukavnikov, Roman},
  title   = {Perverse sheaves on affine flags and {Langlands} dual group},
  journal = {Israel Journal of Mathematics},
  volume  = {170},
  pages   = {135--183},
  year    = {2009},
  doi     = {10.1007/s11856-009-0024-y},
  note    = {With an appendix by Bezrukavnikov and Ivan Mirkovi{\'c}}
}

@article{Alv19,
  author  = {Alvarenga, Roberto},
  title   = {{O}n graphs of {H}ecke operators},
  journal = {Journal of Number Theory},
  year    = {2019},
  volume  = {199},
  pages   = {192--228}
}

@article{AB24,
  author  = {Alvarenga, Roberto and Bonnel, Nans},
  title   = {{H}ecke eigenspaces for the projective line},
  journal = {Journal of Number Theory},
  year    = {2024},
  volume  = {264},
  pages   = {59--98}
}

@article{Bez16,
  author  = {Bezrukavnikov, Roman},
  title   = {On two geometric realizations of an affine {Hecke} algebra},
  journal = {Publications mathématiques de l'IHÉS},
  volume  = {123},
  pages   = {1--67},
  year    = {2016},
  doi     = {10.1007/s10240-015-0077-x}
}

@article{Bos22,
  author  = {uit de Bos, Niels},
  title   = {An Explicit Geometric Langlands Correspondence for the Projective Line Minus Four Points},
  journal = {International Mathematics Research Notices},
  volume  = {2022},
  number  = {24},
  year    = {2022},
  pages   = {19690--19746},
  doi     = {10.1093/imrn/rnab218}
}

@incollection{BZN18,
  author    = {Ben-Zvi, David and Nadler, David},
  title     = {Betti geometric {Langlands}},
  booktitle = {Algebraic Geometry: Salt Lake City 2015},
  series    = {Proceedings of Symposia in Pure Mathematics},
  volume    = {97.2},
  pages     = {3--41},
  publisher = {American Mathematical Society},
  address   = {Providence, RI},
  year      = {2018},
  doi       = {10.1090/pspum/097.2/01},
  eprint    = {1606.08523},
  archivePrefix = {arXiv},
  primaryClass  = {math.AG}
}

@article{Dri87,
  author  = {Drinfeld, V. G.},
  title   = {Two-dimensional {$\ell$}-adic representations of the Galois group of a global field of characteristic {$p$} and automorphic forms on {GL}(2)},
  journal = {Journal of Soviet Mathematics},
  volume  = {36},
  number  = {1},
  year    = {1987},
  pages   = {93--105},
  doi     = {10.1007/BF01104975},
  note    = {English translation of the 1984 Russian original}
}

@article{Hei04,
  author  = {Heinloth, Jochen},
  title   = {Coherent sheaves with parabolic structure and construction of {H}ecke eigensheaves for some ramified local systems},
  journal = {Annales de l'Institut Fourier},
  volume  = {54},
  number  = {7},
  year    = {2004},
  pages   = {2235--2325},
  doi     = {10.5802/aif.2080}
}

@article{HNY13,
  author  = {Heinloth, Jochen and Ng{\^o}, Bao Ch{\^a}u and Yun, Zhiwei},
  title   = {Kloosterman sheaves for reductive groups},
  journal = {Annals of Mathematics},
  series  = {2},
  volume  = {177},
  number  = {1},
  year    = {2013},
  pages   = {241--310},
  doi     = {10.4007/annals.2013.177.1.5}
}

@article{Li79,
  title = {Eisenstein series and decomposition theory over function fields},
  volume = {240},
  ISSN = {1432-1807},
  url = {http://dx.doi.org/10.1007/BF01364628},
  DOI = {10.1007/bf01364628},
  number = {2},
  journal = {Mathematische Annalen},
  publisher = {Springer Science and Business Media LLC},
  author = {Li,  Wen-Ch’ing Winnie},
  year = {1979},
  month = June,
  pages = {115–139}
}

@article{Lor13,
  author  = {Lorscheid, Oliver},
  title   = {{G}raphs of {H}ecke operators},
  journal = {Algebra \& Number Theory},
  year    = {2013},
  volume  = {7},
  number  = {1},
  pages   = {19--61}
}

@article{KZ26,
      title={Eigenforms and graphs of {H}ecke operators with wild ramification}, 
      author={Rudrendra Kashyap and Vladyslav Zveryk},
      year={2026},
      eprint={2603.15931},
      archivePrefix={arXiv},
      primaryClass={math.AG}
}

@article{Schieder2015, 
author = {Schieder, Simon}, 
title = {The {H}arder--{N}arasimhan stratification of the moduli stack of {G}-bundles via {D}rinfeld's compactifications}, 
journal = {Selecta Mathematica}, 
volume = {21}, 
number = {3}, 
pages = {763--831}, 
year = {2015} 
}

@misc{gaitsgory2016recentprogressgeometriclanglands,
      title={Recent progress in geometric Langlands theory}, 
      author={Dennis Gaitsgory},
      year={2016},
      eprint={1606.09462},
      archivePrefix={arXiv},
      primaryClass={math.AG},
      url={https://arxiv.org/abs/1606.09462}, 
}

@article{Yun16,
  author  = {Yun, Zhiwei},
  title   = {Epipelagic representations and rigid local systems},
  journal = {Selecta Mathematica},
  series  = {N.S.},
  volume  = {22},
  number  = {3},
  year    = {2016},
  pages   = {1157--1193},
  doi     = {10.1007/s00029-015-0204-z}
}

@incollection{Yun14b,
  author    = {Yun, Zhiwei},
  title     = {Rigidity in automorphic representations and local systems},
  booktitle = {Current Developments in Mathematics 2013},
  pages     = {73--168},
  publisher = {International Press},
  address   = {Somerville, MA},
  year      = {2014},
  doi       = {10.4310/cdm.2013.v2013.n1.a2},
  eprint    = {1405.3035},
  archivePrefix = {arXiv},
  primaryClass  = {math.AG}
}
\end{document}